\pdfoutput=1 
\documentclass{article}

\usepackage{amsfonts,amsmath, amsthm, amssymb}
\usepackage[latin1]{inputenc}
\usepackage[english]{babel}
\usepackage{graphicx}
\usepackage{enumerate}

\usepackage{tikz}
\usetikzlibrary{decorations.pathmorphing, arrows, calc, patterns}
\tikzset{
    partial ellipse/.style args={#1:#2:#3}{
        insert path={+ (#1:#3) arc (#1:#2:#3)}
    }
}

\newtheorem{theorem}{Theorem}
   \newtheorem*{TTThm}{Theorem~\ref{ToTThm}}
   \newtheorem*{TDThm}{Theorem~\ref{TTDThm}}
   \newtheorem*{FnThm}{Theorem~\ref{Fn}}
   \newtheorem*{Fduality}{Theorem~\ref{submodular duality}}
\newtheorem{lemma}[theorem]{Lemma}

\theoremstyle{definition}

\newcommand{\OO}{\mathcal O}

\newcommand{\Po}{\mathcal P}

\newcommand{\V}{\mathcal V}

\newcommand{\T}{\mathcal T}
\newcommand{\N}{\mathcal N}

\newcommand{\F}{\mathcal{F}}

\newcommand{\sub}{\subseteq}

\def\lowfwd #1#2#3{{\mathop{\kern0pt #1}\limits^{\kern#2pt\raise.#3ex
\vbox to 0pt{\hbox{$\scriptscriptstyle\rightarrow$}\vss}}}}
\def\lowbkwd #1#2#3{{\mathop{\kern0pt #1}\limits^{\kern#2pt\raise.#3ex
\vbox to 0pt{\hbox{$\scriptscriptstyle\leftarrow$}\vss}}}}
\def\fwd #1#2{{\lowfwd{#1}{#2}{15}}}
\def\ve{\kern-1pt\lowfwd e{1.5}2\kern-1pt}
\def\ev{\kern-1pt\lowbkwd e{1.5}2\kern-1pt}

\def\vr{\lowfwd r{1.5}2}
\def\rv{\lowbkwd r02}
\def\vv{\lowfwd v{1.5}2}
\def\vleft{\lowbkwd v02}
\def\vu{\lowfwd u{1.5}2}
\def\uv{\lowbkwd u02}
\def\vw{\lowfwd w{1.5}2}

\def\vrdash{{\mathop{\kern0pt r\lower.5pt\hbox{${}
     \scriptstyle'$}}\limits^{\kern0pt\raise.02ex
     \vbox to 0pt{\hbox{$\scriptscriptstyle\rightarrow$}\vss}}}}
\def\rvdash{{\mathop{\kern0pt r\lower.5pt\hbox{${}
     \scriptstyle'$}}\limits^{\kern0pt\raise.02ex
     \vbox to 0pt{\hbox{$\scriptscriptstyle\leftarrow$}\vss}}}}
\def\vrone{\lowfwd {r_1}12}
\def\rvone{\lowbkwd {r_1}02}
\def\vrtwo{\lowfwd {r_2}12}
\def\rvtwo{\lowbkwd {r_2}02}
\def\vs{\lowfwd s{1.5}1}
\def\sv{\lowbkwd s{1.5}1}

\def\vsidash{{\mathop{\kern0pt s_i\kern-3.5pt\lower.3pt\hbox{${}
     \scriptstyle'$}}\limits^{\kern0pt\raise.02ex
     \vbox to 0pt{\hbox{$\scriptscriptstyle\rightarrow$}\vss}}}}
\def\vS{{\hskip-1pt{\fwd S3}\hskip-1pt}} 
\def\vSk{\lowfwd {S_k}11}
\def\vSr{{\vec S}_{\raise.1ex\vbox to 0pt{\vss\hbox{$\scriptstyle\ge\vr$}}}}
\def\vSdash{{\mathop{\kern0pt S\lower-1pt\hbox{${}
     \scriptstyle'$}}\limits^{\kern2pt\raise.1ex
     \vbox to 0pt{\hbox{$\scriptscriptstyle\rightarrow$}\vss}}}}
\def\vsdash{{\mathop{\kern0pt s\lower.5pt\hbox{${}
     \scriptstyle'$}}\limits^{\kern0pt\raise.02ex
     \vbox to 0pt{\hbox{$\scriptscriptstyle\rightarrow$}\vss}}}}
     
\def\svdash{{\mathop{\kern0pt s\lower.5pt\hbox{${}
     \scriptstyle'$}}\limits^{\kern0pt\raise.02ex
     \vbox to 0pt{\hbox{$\scriptscriptstyle\leftarrow$}\vss}}}}
     
\def\vtdash{{\mathop{\kern0pt t\lower0pt\hbox{${}
     \scriptstyle'$}}\limits^{\kern0pt\raise.1ex
     \vbox to 0pt{\hbox{$\scriptscriptstyle\rightarrow$}\vss}}}}
     
\def\tvdash{{\mathop{\kern0pt t\lower0pt\hbox{${}
     \scriptstyle'$}}\limits^{\kern0pt\raise.1ex
     \vbox to 0pt{\hbox{$\scriptscriptstyle\leftarrow$}\vss}}}}
     
\def\vddash{{\mathop{\kern0pt d\raise1pt\hbox{${}
     \scriptstyle'$}}\limits^{\kern0pt\raise.02ex
     \vbox to 0pt{\hbox{$\scriptscriptstyle\rightarrow$}\vss}}}}
     
\def\dvdash{{\mathop{\kern0pt d\raise1pt\hbox{${}
     \scriptstyle'$}}\limits^{\kern0pt\raise.02ex
     \vbox to 0pt{\hbox{$\scriptscriptstyle\leftarrow$}\vss}}}}
     
\def\vtstar{{\mathop{\kern0pt t\raise2.5pt\hbox{${}
     \scriptstyle*$}}\limits^{\kern0pt\raise.1ex
     \vbox to 0pt{\hbox{$\scriptscriptstyle\rightarrow$}\vss}}}}
     
\def\tvstar{{\mathop{\kern0pt t\raise2.5pt\hbox{${}
     \scriptstyle*$}}\limits^{\kern0pt\raise.1ex
     \vbox to 0pt{\hbox{$\scriptscriptstyle\leftarrow$}\vss}}}}

\def\vtstarD{{\mathop{\kern0pt t\kern.5pt\raise3pt\hbox{${}
     \scriptstyle*$}{\kern-5.5pt\lower3pt\hbox{$
     \scriptstyle D$}}}\limits^{\kern0pt\raise.1ex
     \vbox to 0pt{\hbox{$\scriptscriptstyle\rightarrow$}\vss}}}}

\def\tvstarD{{\mathop{\kern0pt t\kern.5pt\raise3pt\hbox{${}
     \scriptstyle*$}{\kern-5.5pt\lower3pt\hbox{$
     \scriptstyle D$}}}\limits^{\kern0pt\raise.1ex
     \vbox to 0pt{\hbox{$\scriptscriptstyle\leftarrow$}\vss}}}}

\def\vt{\lowfwd t{1.5}1}
\def\tv{\lowbkwd t{1.5}1}
\def\vT{{\fwd T1}}
\def\vU{{\vec U}} 

\def\vt{\lowfwd t{1.5}1}
\def\tv{\lowbkwd t{1.5}1}

\def\vx{\lowfwd x{1.5}1}
\def\xv{\lowbkwd x{1.5}1}
\def\vy{\lowfwd y{1.5}1}
\def\yv{\lowbkwd y{1.5}1}

\def\?#1{\vadjust{\vbox to 0pt{\vss\vskip-8pt\leftline{%
     \llap{\hbox{\vbox{\pretolerance=-1
     \doublehyphendemerits=0\finalhyphendemerits=0
     \hsize16truemm\tolerance=10000\small
     \lineskip=0pt\lineskiplimit=0pt
     \rightskip=0pt plus16truemm\baselineskip8pt\noindent
     \hskip0pt        
     #1\endgraf}\hskip7truemm}}}\vss}}}
\def\COMMENT#{}

    \title{Structural submodularity and tangles in abstract separation systems}
	\author{Reinhard Diestel, Joshua Erde and Daniel Wei\ss{}auer}     
	\date{}
     
\begin{document}
	
	     \maketitle
     
	\begin{abstract}
		We prove a tree-of-tangles theorem and a tangle-tree duality theorem for abstract separation systems~$\vS$ that are submodular in the structural sense that, for every pair of oriented separations, $\vS$~contains either their meet or their join defined in some universe~$\vU$ of separations containing~$\vS$.

This holds, and is widely used, if $\vU$ comes with a submodular order function and $\vS$ consists of all its separations up to some fixed order. Our result is that for the proofs of these two theorems, which are central to abstract tangle theory, it suffices to assume the above structural consequence for~$\vS$, and no order function is needed.
\end{abstract}	     
     
     \begin{section}{Introduction}
     	
This paper is, in a sense, the capstone of a comprehensive project~\cite{CDHH13CanonicalAlg, CDHH13CanonicalParts, confing, CG14:isolatingblocks, AbstractSepSys, TreeSets, ProfileDuality, ProfilesNew, TangleTreeAbstract, TangleTreeGraphsMatroids, JoshRefining, JoshUnified} whose aim has been to utilize the idea of tangles familiar from Robertson and Seymour's graph minors project as a way of capturing clusters in other contexts, such as image analysis~\cite{MonaLisa}, genetics~\cite{TanglesEmpirical}, or the social sciences~\cite{TanglesSocial}. The idea is to use tangles, which in graphs are certain consistent ways of orienting their low-order separations, as an indirect way of capturing `fuzzy' clusters~-- ones that cannot easily be described by simply listing their elements~-- by instead orienting all those low-order separations towards them. We can then think of these as a collection of signposts all pointing to that cluster, and of clusters as collective targets of such consistent pointers.

Once clusters have been captured by `abstract tangles' in this way, one can hope to generalize to such clusters Robertson and Seymour's two fundamental results about tangles in graphs~\cite{GMX}. One of these is the {\em tree-of-tangles\/} theorem. It says that any set of distinguishable tangles~-- ones that pairwise do not contain each other~-- can in fact be distinguished pairwise by a small and nested set of separations: for every pair of tangles there is a separation in this small and nested collection that distinguishes them. Formally, this means that these two tangles orient it differently; informally it means that one of its two orientations points to one of the tangles, while its other orientation points to the other tangle. Since these separations are nested, they split the underlying structure in a tree-like way, giving it a rough overall structure.

The other fundamental result from~\cite{GMX}, the {\em tangle-tree duality\/} theorem, tells us that if there are no tangles of a desired type then the entire underlying structure can be split in such a tree-like way, i.e.\ by some nested set of separations, so that the regions corresponding to a node of the structure tree are all small. (What exactly this means may%
   \COMMENT{}
   depend on the type of tangle considered.)

This research programme required a number of steps, of which this paper constitutes the last.

The first step was to make the notion of tangles independent from their natural habitat of graphs. In a graph, tangles are ways of consistently orienting all its separations $\{A,B\}$ up to some given order, either as $(A,B)$ or as~$(B,A)$. If we want to do this for another kind of underlying structure than a graph, this structure will have to come with a notion of `separation', it must be possible to `orient' these separations, and there must be a difference between doing this `consistently' or `inconsistently'. If we wish to express, and perhaps prove, the two fundamental tangle theorems in such an abstract context, we further need a notion of when two `separations' are nested.

There are many structures that come with a natural notion of separation. For sets, for example, we might simply take bipartitions. The notion of nestedness can then be borrowed from the nestedness of sets and applied to the bipartition classes. Thinking of a bipartition as an unordered pair of subsets, we can also naturally orient it `towards one or the other of these subsets' by ordering the pair. Finally, we have to come up with natural notions of when orientations of different separations are consistent: we think of this as `roughly pointing the same way', and it is another prerequisite for defining tangles to make this formal. This is both trickier to do in an abstract context and one of our main sources of freedom; we shall address this question in Section~\ref{sec:Defs}.

The completion of the first step in our research programme thus consisted in abstracting from the various notions of separation, and of consistently orienting separations, a minimum set of requirements that might serve as axioms for an abstract notion of tangle applicable to all of them. This resulted in the concept of {\em separation systems\/} and their (`abstract') tangles~\cite{AbstractSepSys}.

The second step, then, was to generalize the proofs of the tree-of-tangles theorem and the tangle-tree duality theorem to the abstract setting of separation systems. This was done in~\cite{ProfilesNew} and~\cite{TangleTreeAbstract}, respectively.

In order to prove these theorems, or to apply them%
   \COMMENT{}
   to concrete cases of abstract separation systems, e.g.\ as in~\cite{TangleTreeGraphsMatroids, MonaLisa}, one so far still needed a further ingredient of graph tangles: a submodular order function on the separation system considered. Our aim in this paper is to show that one can do without this: we shall prove that a structural consequence of the existence of a submodular order function, a~consequence  that can be expressed in terms of abstract separation systems, can replace the assumption that such a function exists in the proofs of the above two theorems. We shall refer to separation systems that satisfy this structural condition as {\em submodular separation systems\/}.%
   \footnote{There is also a notion of submodularity for separation {\em universes\/}. Separation universes are special separation systems that are particularly large, and they are always submodular as separation systems. For separation universes, therefore, submodularity is used with the narrower meaning of being endowed with a submodular order function~\cite{AbstractSepSys}.}

With this third step, then, the programme sketched above will be complete: we shall have a notion of tangle for very general abstract separation systems, as well as a tree-of-tangles theorem and a tangle-tree duality theorem for these tangles that can be expressed and proved without the need for any submodulary order function on the separation systems considered.

Formally, our two main results read as follows:

     	\begin{theorem} \label{ToTThm}
     		Every submodular separation system~${\vS}$ contains a tree set of separations that distinguishes all the abstract tangles of~$S$.
     	\end{theorem}
     	
\begin{theorem} \label{TTDThm}
     		Let~${\vS}$ be a submodular separation system without degenerate elements in a distributive universe~${\vU}$. Then exactly one of the following holds:\vskip-\medskipamount\vskip-\medskipamount
     		\begin{enumerate}[\rm (i)]\itemsep=0pt
     			\item $S$ has an abstract tangle.
     			\item There exists an $S$-tree over~$\T^*$ (witnessing that $S$ has no abstract tangle).
     		\end{enumerate}
     	\end{theorem}

\noindent
(See Section~\ref{sec:Defs} for definitions.) Three further theorems, which partly strengthen or generalize the above two, will be stated in Section~\ref{sec:Defs} (and proved later) when we have more terminology available.

One may ask, of course, whether weakening the existence of a submodular order function to `structural submodularity' in the premise of these two theorems is worth the effort. We believe it is. For a start, the entire programme of developing abstract separation systems, and  a theory of tangles for them, served the purpose of identifying the few structural assumptions one has to make of a set of objects called `separations' in order to capture the essence of tangles in graphs, and thereby make them applicable in much wider contexts. It would then seem oblivious of these aims to stop just short of the goal: to continue to make unnecessarily strong assumptions of an extraneous and non-structural kind when weaker structural assumptions can achieve the same.

However, there is also a technical advantange. As we shall see in Sections~\ref{sec:cliques} and~\ref{sec:phylo}, there are interesting abstract separation systems that are structurally submodular but which do not come with a natural submodular order function that implies this.

\end{section}

\begin{section}{Abstract separation systems}\label{sec:Defs}
     	
     	Abstract separation systems were first introduced in~\cite{AbstractSepSys}; see there for a gentle formal introduction and any terminology we forgot to define below. Motivation for why they are interesting can be found in the introductory sections of \cite{ProfilesNew, TangleTreeAbstract, TangleTreeGraphsMatroids} and in~\cite{MonaLisa}. In what follows we provide a self-contained account of just the definitions and basic facts about abstract separation systems that we need in this paper.

     	A \emph{separation system} $({\vS}, \leq,\! {}^*)$ is a partially ordered set with an order-reversing involution $^*\!: {\vS} \to {\vS}$. The elements of~${\vS}$ are called \emph{(oriented) separations}. The \emph{inverse} of ${\vs} \in {\vS}$ is~${\vs}{}^*$, which we usually denote by~${\sv}$.%
   \vadjust{\penalty-200}
   An \emph{(unoriented) separation} is a set $s = \{ {\vs}, {\sv} \}$ consisting of a separation and its inverse and we then refer to~${\vs}$ and~${\sv}$ as the two \emph{orientations} of~$s$. Note that it may occur that ${\vs} = {\sv}$, we then call~${\vs}$ \emph{degenerate}. The set of all separations is denoted by~$S$. When the context is clear, we often refer to oriented separations simply as separations in order to improve the flow of text.

     	If the partial order $({\vS}, \leq)$ is a lattice with join~$\vee$ and meet~$\wedge$, then we call $({\vS}, \leq,\! {}^*, \vee, \wedge)$ a \emph{universe} of (oriented) separations. It is \emph{distributive} if it is distributive as a lattice. Typically, the separation systems we are interested in are contained in a universe of separations. In most applications, one starts with a universe $({\vU}, \leq, {}^*, \vee, \wedge)$ and then defines~${\vS}$ as a set of separations of low order with respect to some {\em order function\/} on~$\vU$, a map $\lvert \, \cdot \, \lvert : {\vU} \to [0,\infty)$ that is \emph{symmetric} in that $| {\vs} | = |{\sv}|$, and  \emph{submodular} in that $|{\vs} \vee {\vt}| + |{\vs} \wedge {\vt}| \leq |{\vs}| + |{\vt}|$ for all ${\vs}, {\vt} \in {\vU}$. Submodularity of the order function in fact plays a crucial role in several arguments. One of its most immediate consequences is that whenever both ${\vs}, {\vt} \in \vSk{}:= \{ {\vu} \in {\vU} \colon |{\vu}| < k \}$, then at least one of ${\vs} \vee {\vt}$ and ${\vs} \wedge {\vt}$ again lies in~${\vSk}$. 
     	
     	In order to avoid recourse to the external concept of an order function if possible, let us turn this last property into a definition that uses only the language of lattices. Let us call a subset~$M$ of a lattice $(L, \vee, \wedge)$ \emph{submodular} if for all $x, y \in M$ at least one of $x \vee y$ and $x \wedge y$ lies in~$M$. A~separation system $\vS$ contained in a given universe~$\vU$ of separations is \emph{(structurally) submodular} if it is submodular as a subset of the lattice underlying~$\vU$.
	     	
     	We say that ${\vs} \in {\vS}$ is \emph{small} (and~${\sv}$ is co-small) if ${\vs} \leq {\sv}$. An element~${\vs}\in\vS$ is \emph{trivial} in~$\vS$ (and~${\sv}$ is co-trivial) if there exists $t\in S$ whose orientations $\vt,\tv$ satisfy ${\vs} < {\vt}$ as well as $\vs < {\tv}$. Notice that trivial separations are small.

     	Two separations $s, t \in S$ are \emph{nested} if there exist orientations ${\vs}$ of~$s$ and ${\vt}$ of~$t$ such that ${\vs} \leq {\vt}$. Two oriented separations are nested if their underlying separations are. We say that two separations \emph{cross} if they are not nested. A set of (oriented) separations is \emph{nested} if any two of its elements are. A nested separation system without trivial or degenerate elements is a \emph{tree set}. A set~$\sigma$ of non-degenerate oriented separations is a \emph{star} if for any two distinct ${\vs}, {\vt} \in \sigma$ we have ${\vs} \leq {\tv}$. A family $\F \sub 2^{{\vU}}$ of sets of separations is \emph{standard for~${\vS}$} if for any trivial ${\vs} \in {\vS}$ we have $\{ {\sv} \} \in \F$. Given $\F \sub 2^{{\vU}}$, we write~$\F^*$ for the set of all elements of~$\F$ that are stars.
     	
     	An \emph{orientation} of~$S$ is a set $O \sub {\vS}$ which contains for every $s \in S$, exactly one of ${\sv}, {\vs}$. An orientation~$O$ of~$S$ is \emph{consistent} if whenever $r, s \in S$ are distinct and ${\vr} \leq {\vs}\in O$, then ${\rv} \notin O$. The idea behind this is that separations~$\rv$ and~$\vs$ are thought of as pointing away from each other if $\vr\le\vs$. If we wish to orient~$r$ and~$s$ towards some common region of the structure which they are assumed to `separate', as is the idea behind tangles, we should therefore not orient them as~$\rv$ and~$\vs$.

Tangles in graphs also satisfy another, more subtle, consistency requirement: they never orient three separations $r,s,t$ so that the region to which they point collectively is `small'.%
   \footnote{Formally: so that the union of their sides to which they do {\em not\/} point is the entire graph.}
    This can be mimicked in abstract separation systems by asking that three oriented separations in an `abstract tangle' must never have a co-small supremum; see~\cite[Section~5]{AbstractSepSys}. So let us implement this formally.

Given a family $\F \sub 2^{{\vU}}$, we say that~$O$ \emph{avoids}~$\F$ if there is no $\sigma \sub O$ with  $\sigma \in \F$. A consistent $\F$-avoiding orientation of~$S$ is called an {\em $\F$-tangle\/} of~$S$. An~$\F$-tangle for $\F = \T$ with
 $$\T:= \{ \{ {\vr}, {\vs}, {\vt} \} \sub {\vU} \colon {\vr} \vee {\vs} \vee {\vt} \text{ is co-small} \}$$
 is an \emph{abstract tangle}.

 A~separation $s \in S$ \emph{distinguishes} two orientations $O_1, O_2$ of $S$ if $O_1 \cap s \neq O_2 \cap s$. Likewise, a set~$N$ of separations \emph{distinguishes} a set~$\OO$ of orientations if for any two $O_1, O_2 \in \OO$, there is some $s \in N$ which distinguishes them.

Let us restate our tree-of-tangles theorem for abstract tangles of submodular separation systems:	
     	
     	\begin{TTThm}
     		Every submodular separation system~${\vS}$ contains a tree set of separations that distinguishes all the abstract tangles of~$S$.
     	\end{TTThm}
     	
     	We now introduce the structural dual to the existence of abstract tangles. An \emph{$S$-tree} is a pair $(T, \alpha)$ consisting of a tree~$T$ and a map $\alpha : {\vec E}(T) \to {\vS}$ from the set ${\vec E}(T)$ of orientations of edges of~$T$ to~${\vS}$ such that $\alpha(y,x) = \alpha(x,y)^*$ for all $xy \in E(T)$. Given $\F \sub 2^{{\vU}}$, we call $(T, \alpha)$ an \emph{$S$-tree over~$\F$} if $\alpha( F_t) \in \F$ for every $t \in T$, where
     	\[
     	F_t := \{ (s,t) \colon st \in E(T) \} .
     	\]     	
It is easy to see that if~${S}$ has an abstract tangle, then there can be no ${S}$-tree over~$\T$.

Our tangle-tree duality theorem for abstract tangles of submodular separation systems, which we now re-state, asserts a converse to this. Recall that~$\T^*$ denotes the set of stars in~$\T$.

     	     	\begin{TDThm}
     		Let~${\vS}$ be a submodular separation system without degenerate elements in a distributive universe~${\vU}$. Then exactly one of the following holds:
     		\begin{enumerate}[\rm (i)]\itemsep=0pt\vskip-\smallskipamount\vskip0pt
     			\item ${S}$ has an abstract tangle.
     			\item There exists an $S$-tree over~$\T^*$.
     		\end{enumerate}
     	\end{TDThm}
     	
	Here, it really is necessary to exclude degenerate separations: a single degenerate separation will make the existence of abstract tangles impossible, although there might still be $\T^*$-tangles (and therefore no $S$-trees over~$\T^*$). We will actually prove a duality theorem for $\T^*$-tangles without this additional assumption and then observe that $\T^*$-tangles are in fact already abstract tangles, unless~${\vS}$ contains a degenerate separation. 
     	
	In applications, we do not always wish to consider all the abstract tangles of a given separation system. For example, if~$S$ consists of the bipartitions $\{A,B\}$ of some finite set~$V$ (see~\cite{AbstractSepSys} for definitions), then every $v \in V$ induces an abstract tangle
     	\[
     	\tau_v := \big\{ (A, B) \in {\vS} \colon v \in B \big\} ,
     	\]
	the \emph{principal tangle} induced by~$v$.%
   \vadjust{\penalty-200}
   In particular, abstract tangles trivially exist in these situations. In order to exclude principal tangles, we could require that every tangle~$\tau$ of~$S$ must satisfy $(\{ v \}, V \setminus \{ v \}) \in \tau$ for every $v \in V\!$.
	
	More generally, we might want to prescribe for some separations~$s$ of~$S$ that any tangle of~$S$ we consider must contain a particular one of the two orientations of~$s$ rather than the other. This can easily be done in our abstract setting, as follows. Given $Q \sub {\vU}$, let us say that an abstract tangle~$\tau$ of~$S$ \emph{extends}~$Q$ if $Q \cap {\vS} \sub \tau$. It is easy to see that~$\tau$ extends~$Q$ if and only if~$\tau$ is $\F_Q$-avoiding, where 
	\[
	\F_Q := \{ \{ {\sv} \} \colon {\vs} \in Q \text{ non-degenerate}\}.
	\]
	 We call $Q \sub\vU$ {\em down-closed\/} if $\vr\le\vs\in Q$ implies $\vr\in Q$ for all $\vr,\vs\in\vU$.
	
	Here, then, is our refined tangle-tree duality theorem for abstract tangles of submodular separation systems.
     	
	     	     	\begin{theorem} \label{TTDThm2}
     		Let~${\vS}$ be a submodular separation system without degenerate elements in a distributive universe~${\vU}$ and let $Q \sub {\vU}$ be down-closed. Then exactly one of the following assertions holds:
     		\begin{enumerate}[\rm (i)]\itemsep=0pt\vskip-\smallskipamount\vskip0pt
     			\item $S$ has an abstract tangle extending~$Q$.
     			\item There exists an ${S}$-tree over $\T^* \cup \F_Q$.
     		\end{enumerate}
     	\end{theorem}

\noindent
   Observe that Theorem~\ref{TTDThm2} implies Theorem~\ref{TTDThm} by taking $Q = \emptyset$. 

\medbreak	

The abstract tangles in Theorem~\ref{TTDThm2} are not the only $\F$-tangles for which such a  statement holds. In~\cite{TangleTreeGraphsMatroids} the tangle-tree duality theorem of~\cite{TangleTreeAbstract} is used to prove such a statement for a broad class of $\F$-tangles, albeit under a stronger assumption: one needs there that~${\vS}$ is not just structurally submodular, as is our assumption here throughout our paper, but that $\vU$ has a submodular order function and $\vS$ is the set separations up to some fixed order (and therefore, in particular, submodular).

In Section~\ref{s:duality}, however, we will show that the weaker assumption that~${\vS}$ itself is submodular is in fact sufficient to establish the only property of~$S$%
   \COMMENT{}
   whose proof in~\cite{TangleTreeGraphsMatroids} requires a submodular order function: this is the fact that $S$ is `separable'. (We shall repeat the definition of this in Section~\ref{s:duality}.)

The other ingredient one needs for all those applications of the tangle-tree duality theorem from~\cite{TangleTreeAbstract} is a property of~$\F$: that $\F$ is `closed under shifting'. Sometimes, a~submodular order function on~$\vU$ is needed also to establish this property of~$\F$. But if it is not, we can now prove the same application without a submodular order function, assuming only that $S$ itself is submodular:

	\begin{theorem}\label{submodular duality}
	Let ${\vU}$ be a universe of separations and ${\vS} \sub {\vU}$ a submodular separation system. Let $\F \sub 2^{{\vU}}$ be a set of stars which is standard for~${\vS}$ and closed under shifting. Then exactly one of the following holds:
	     		\begin{enumerate}[\rm (i)]\itemsep=0pt\vskip-\smallskipamount\vskip0pt
	     			\item There exists an $\F$-tangle of~$S$.
	     			\item There exists an $S$-tree over~$\F$.
	     		\end{enumerate}
	     	\end{theorem}
 \noindent
We shall prove Theorem~\ref{submodular duality} in Section~\ref{s:duality}.

\medbreak

Our last result is an example of Theorem~\ref{submodular duality} for a concrete~$\F$, a tangle-tree duality theorem for $\F$-tangles of bipartitions of a set that are used particularly often in applications~\cite{TanglesEmpirical,TanglesSocial}. Let ${\vU}$ be the universe of oriented bipartitions $(A,B)$ of a set~$V\!$ (see~\cite{AbstractSepSys} for definitions). Let $m\ge 1$ and $n\ge 2$%
   \COMMENT{}
    be integers, and define
$$\F_m := \big\{\,F\sub\vU : \big|\!\bigcap_{(A,B)\in F}\!\!\! B\>\big| < m\, \big\}$$
\vskip-.5\baselineskip\noindent
and
 $$\F^n_m := \{\,F\in\F_m : |F| < n\,\}.$$
 To subsume $\F_m$ under this latter notation we allow $n=\infty$, so that $\F^\infty_m = \F_m$.
 Given any collection $\F\sub 2^\vU$ of sets of oriented separations, we write~$\F^*$ for its subcollection of those sets $F\in\F$ that are stars (of oriented separations).

We shall prove in Section~\ref{s:duality} that the set of stars in $\F_m$ is closed under shifting. Building on Theorem~\ref{submodular duality}, we then use this in Section~\ref{apps} to prove the following:

     	\begin{theorem}\label{Fn}
     	Let ${\vS} \sub {\vU}$ be a submodular separation system, let $1\le m\in{\mathbb N}$ and $2\le n\in{\mathbb N}\cup\{\infty\}$, and let $\F = \F^n_m$.%
   \COMMENT{}
   Then exactly one of the following two statements holds:
     	\begin{enumerate}[\rm (i)]\itemsep=0pt\vskip-\smallskipamount\vskip0pt
     	\item $S$ has an $\F$-tangle;
     	\item There exists an $S$-tree over $\F^*$.
     	\end{enumerate}
      	\end{theorem}

\medskip
   The bound $n$ on the size of the sets in~$\F$ is often taken to be~4.%
   \COMMENT{}
   In~(i) we could replace $\F$ with~$\F^*$, since for these~$\F$ the $\F$-tangles are precisely the~$\F^*$-tangles; see Section~\ref{sec:cluster}.

\end{section}

     \begin{section}{The tree-of-tangles theorem}\label{sec:tree}
     	In this section we will prove Theorem~\ref{ToTThm}.	In fact, we are going to prove a slightly more general statement. Let $\Po := \{ \{ {\vs}, {\vt}, ({\vs} \vee {\vt})^* \} \colon {\vs}, {\vt} \in {\vU} \}$. The $\Po$-tangles are known as \emph{profiles}. A profile of $S$ is {\em regular\/} if it contains all the small separations in~$\vS$.

     	\begin{theorem} \label{tangle tree reg prof}
     		Let~${\vS}$ be a submodular separation system and~$\Pi$ a set of profiles of~$S$. Then~$\vS$ contains a tree set that distinguishes~$\Pi$.
     	\end{theorem}
     	
	This implies Theorem~\ref{ToTThm}, by the following easy observation.     	
     	
     	\begin{lemma}
     		Every abstract tangle is a profile.
     	\end{lemma}
     		
     	\begin{proof}
     		Let ${\vs}, {\vt} \in \vU$ and ${\vr} := {\vs} \vee {\vt}$. Then
     		\[
     		{\vs} \vee {\vt} \vee {\rv} = {\vr} \vee {\rv}
     		\]
     		is co-small, so $\{ {\vs}, {\vt}, {\rv} \} \in \T$. Therefore $\Po \sub \T$ and every $\T$-tangle is also a $\Po$-tangle.
     	\end{proof}

     	We first recall a basic fact about nestedness of separations. 	For $s, t\in S$, we define the \emph{corners} ${\vs} \wedge {\vt}$, ${\vs} \wedge {\tv}$, ${\sv} \wedge {\vt}$ and ${\sv} \wedge {\tv}$. 
     	
	\begin{lemma}[\cite{AbstractSepSys}] \label{fish lemma}
			Let~${\vS}$ be a separation system in a universe~${\vU}$ of separations. Let $s, t$ be two crossing separations and~${\vr}$ one of the corners. Then every separation that is nested with both~$s$ and~$t$ is nested with~$r$ as well.
	\end{lemma}     	
	
     	
     	In the proof of Theorem~\ref{tangle tree reg prof}, we take a nested set~$\N$ of separations that distinguishes some set~$\Pi_0$ of regular profiles and we want to exchange one element of~$\N$ by some other separation while maintaining that~$\Pi_0$ is still distinguished. The following lemma simplifies this exchange.
     	
     	\begin{lemma} \label{unique pair distinguished}
     		Let~${\vS}$ be a separation system, $\OO$ a set of consistent orientations of~${S}$ and $\N \sub S$ an inclusion-minimal nested set of separations that distinguishes~$\OO$. Then for every $t \in \N$ there is a unique pair of orientations $O_1, O_2 \in \OO$ that are distinguished by~$t$ and by no other element of~$\N$.
     		\end{lemma}
     	
     	\begin{proof}
     	It is clear that at least one such pair must exist, for otherwise $\N \setminus \{ t \}$ would still distinguish~$\OO$, thus violating the minimality of~$\N$. 
     	
     	 Suppose there was another such pair, say $O_1', O_2'$. After relabeling, we may assume that ${\vt} \in O_1 \cap O_1'$ and ${\tv} \in O_2 \cap O_2'$. By symmetry, we may further assume that $O_1 \neq O_1'$. Since~$\N$ distinguishes~$\OO$, there is some $r \in \N$ with ${\vr} \in O_1$, ${\rv} \in O_1'$.
     	
     	As~$t$ is the only element of~$\N$ distinguishing $O_1, O_2$, it must be that ${\vr} \in O_2$ as well, and similarly ${\rv} \in O_2'$. We hence see that for any orientation~$\tau$ of $\{ r,t \}$, there is an $O \in \{ O_1, O_2, O_1', O_2' \}$ with $\tau \sub O$. Since~$\N$ is nested, there exist orientations of~$r$ and~$t$ pointing away from each other. But then one of $O_1, O_2, O_1', O_2'$ is inconsistent, which is a contradiction.
     	\end{proof}

     	\begin{proof}[Proof of Theorem~\ref{tangle tree reg prof}]
Note that it suffices to show that there is a nested set~$\N$ of separations that distinguishes~$\Pi$: Every consistent orientation contains every trivial and every degenerate element, so any inclusion-minimal such set~$\N$ gives rise to a tree-set.

We prove this by induction on $|\Pi|$, the case $|\Pi| = 1$ being trivial.
	
	For the induction step, let $P \in \Pi$ be arbitrary and $\Pi_0 := \Pi \setminus \{ P \}$. By the induction hypothesis, there exists a nested set~$\N$ of separations that distinguishes~$\Pi_0$. If some such set~$\N$ distinguishes~$\Pi$, there is nothing left to show. Otherwise, for every nested $\N \sub S$ which distinguishes~$\Pi_0$ there is a $P' \in \Pi_0$ which~$\N$ does not distinguish from~$P$. Note that~$P'$ is unique. For any $s \in S$ that distinguishes~$P$ and~$P'$, let $d(\N, s)$ be the number of elements of~$\N$ which are not nested with~$s$.
	
	Choose a pair $(\N, s)$ so that $d(\N, s)$ is minimum. Clearly, we may assume~$\N$ to be inclusion-minimal with the property of distinguishing~$\Pi_0$. If $d(\N, s) = 0$, then $\N \cup \{ s \}$ is a nested set distinguishing~$\Pi$ and we are done, so we now assume for a contradiction that $d( \N, s) > 0$. 
	
	Since~$\N$ does not distinguish $P$ and $P'$, we can fix an orientation of each $t \in \N$ such that ${\vt} \in P \cap P'$. Choose a $t \in \N$ such that~$t$ and~$s$ cross and~${\vt}$ is minimal. Let $(P_1, P_2)$ be the unique pair of profiles in~$\Pi_0$ which are distinguished by~$t$ and by no other element of~$\N$, say ${\tv} \in P_1$, ${\vt} \in P_2$. Let us assume without loss of generality that~${\sv} \in P_1$. The situation is depicted in Figure~\ref{pic: tangle tree thm}. Note that we do not know whether ${\vs} \in P_2$ or ${\sv} \in P_2$. Also, the roles of~$P$ and~$P'$ might be reversed, but this is insignificant.
	
	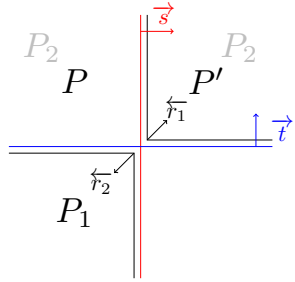
\begin{figure}[ht!]
	
	\begin{center}
	\resizebox{4cm}{!}{
\begin{tikzpicture}
\draw[thick] (3.8,0)--(3.8,3.8)--(0,3.8);
\draw[->,thick] (3.8,3.8)--(3.2,3.2);
\node[scale=2] at (2.8,2.9) {$\overleftarrow{r_2}$};
\draw[thick] (4.2,8)--(4.2,4.2)--(8,4.2);
\draw[->,thick] (4.2,4.2)--(4.8,4.8);
\node[scale=2] at (5.1,5.1) {$\overleftarrow{r_1}$};
\draw[red,thick] (4,0)--(4,8);
\draw[blue,thick] (0,4)--(8,4);
\draw[red,->,thick] (4,7.5)--(5,7.5);
\node[text=red,scale=2] at (4.7,8) {$\overrightarrow{s}$};
\draw[blue,->,thick] (7.5,4)--(7.5,5);
\node[text=blue,scale=2] at (8.3,4.5) {$\overrightarrow{t}$};

\node[scale=3] at (2,6) {$P$};
\node[scale=3] at (6,6) {$P'$};
\node[scale=3] at (2,2) {$P_1$};
\node[scale=3,lightgray] at (1,7) {$P_2$};
\node[scale=3,lightgray] at (7,7) {$P_2$};
\end{tikzpicture}
}

\end{center}
    \vskip-12pt
	\caption{Crossing separations}
	\label{pic: tangle tree thm}
\end{figure}
	
	Suppose first that ${\vrone} := {\vs} \vee {\vt} \in {\vS}$. Let $Q \in \{ P, P' \}$. If ${\vs} \in Q$, then ${\vrone} \in Q$, since ${\vt} \in P \cap P'$ and~$Q$ is a profile. If ${\vrone} \in Q$, then ${\vs} \in Q$ since~$Q$ is consistent and ${\vs} \leq {\vrone} \in Q$: it cannot be that ${\vs} = {\rvone}$, since then~$s$ and~$t$ would be nested. Hence each $Q \in \{ P, P'  \}$ contains~${\vrone}$ if and only if it contains~${\vs}$. In particular, $r_1$ distinguishes~$P$ and~$P'$. By Lemma~\ref{fish lemma}, every $u \in \N$ that is nested with~$s$ is also nested with~$r_1$. Moreover, $t$ is nested with~$r_1$, but not with~$s$, so that $d(\N, r_1) < d(\N, s)$. This contradicts our choice of~$s$.
	
	Therefore ${\vs} \vee {\vt} \notin {\vS}$. Since~${\vS}$ is submodular, it follows that ${\vrtwo} := {\vs} \wedge {\vt} \in {\vS}$. Moreover, $r_2$ is nested with every $u \in \N \setminus \{ t \}$. This is clear if ${\vt} \leq {\vu}$ or ${\vt} \leq {\uv}$, since ${\vrtwo} \leq {\vt}$. It cannot be that ${\uv} \leq {\vt}$, because ${\vu}, {\vt} \in P$ and~$P$ is consistent. Since~$\N$ is nested, only the case ${\vu} < {\vt}$ remains. Then, by our choice of~${\vt}$, $u$ and~$s$ are nested and it follows from Lemma~\ref{fish lemma} that~$u$ and~$r_2$ are also nested. Hence $\N' := (\N \setminus \{ t \}) \cup \{ r_2 \}$ is a nested set of separations. 
	
	To see that~$\N'$ distinguishes~$\Pi_0$, it suffices to check that~$r_2$ distinguishes~$P_1$ and~$P_2$. We have ${\vrtwo} \in P_2$ since~$P_2$ is consistent and ${\vrtwo} \leq {\vt} \in P_2$: if ${\vrtwo} = {\tv}$, then~$s$ and~$t$ would be nested. Since ${\rvtwo} = {\sv} \vee {\tv}$ and ${\sv}, {\tv} \in P_1$, we find ${\rvtwo} \in P_1$. Any element of~$\N'$ which is not nested with~$s$ lies in~$\N$. Since $t \in \N \setminus \N'$ is not nested with~$s$, it follows that $d( \N', s)  < d( \N, s)$, contrary to our choice of~$\N$ and~$s$.
     	\end{proof}
     	
     	\end{section}

		 	\begin{section}{Tangle-tree duality}\label{s:duality}

	     	Our agenda for this section is first to prove Theorem~\ref{submodular duality}, and then to derive from it Theorem~\ref{TTDThm2}, which as we have seen implies Theorem~\ref{TTDThm}. Our proof will be an application of the basic tangle-tree duality theorem from~\cite{TangleTreeAbstract}.

For this we need to introduce the notion of separability, and then prove that submodular separation systems are separable (Lemma~\ref{submod sep sys}). This lemma not only lies at the heart of our proof of Theorem~\ref{submodular duality}: it will also be central to any other result that asserts a tangle-tree type duality for separation systems $({\vS}, \leq,\! {}^*)$ that are structurally submodular, but are not so simply as a corollary of the existence of a submodular order function on~$\vS$.
	     	
	     	A separation ${\vs} \in {\vS}$ \emph{emulates~${\vr}$ in~${\vS}$} if ${\vs} \geq {\vr}$ and for every ${\vt} \in {\vS} \setminus \{ {\rv} \}$ with ${\vt} \geq {\vr}$ we have ${\vs} \vee {\vt} \in {\vS}$. For ${\vs} \in {\vS}$, $\sigma \sub {\vS}$ and ${\vx} \in \sigma$, define     	
     		\[
     		\sigma\!_{{\vx}}^{{\vs}} :=	\{ {\vx} \vee {\vs} \} \cup \{ {\vy} \wedge {\sv} \colon {\vy} \in \sigma \setminus \{ {\vx} \} \} .
     		\]
     		
     		\begin{lemma} \label{shift is star}
     		Suppose ${\vs} \in {\vS}$ emulates a non-trivial ${\vr}$ in~${\vS}$, and let $\sigma \sub {\vS}$ be a star such that ${\vr} \leq {\vx} \in \sigma$. Then $\sigma\!^{{\vs}}_{{\vx}} \sub {\vS}$ is a star.
     		 \end{lemma}
     		
     		\begin{proof}	
     		
     	Note that for every ${\vy} \in \sigma \setminus \{ {\vx} \}$ we have ${\vr} \leq {\yv}$. It is clear that for any two distinct~$\vu, \vv \in \sigma\!^{{\vs}}_{{\vx}}$ we have $\vu \leq \vleft$, so we only need to show that every element of~$\sigma\!^{{\vs}}_{{\vx}}$ is non-degenerate and lies in~${\vS}$. For every ${\vu} \in \sigma\!^{{\vs}}_{{\vx}}$ there is a non-degenerate ${\vt} \in {\vS}$ with ${\vr} \leq {\vt}$ such that either ${\vu} = {\vt} \vee {\vs}$ or ${\uv} = {\vt} \vee {\vs}$.
     		
     	Let ${\vt} \in {\vS}$ be non-degenerate with ${\vr} \leq {\vt}$. Since~${\vs}$ emulates~${\vr}$ in~${\vS}$, we find ${\vt} \vee {\vs} \in {\vS}$. Assume for a contradiction that ${\vt} \vee {\vs}$ was degenerate. Since~${\vt}$ is non-degenerate, we find that ${\vt}  < {\vt} \vee {\vs}$, so that~${\vt}$ is trivial. But then so is~${\vr}$, because ${\vr} \leq {\vt}$. This contradicts our assumption on~${\vr}$.
     		\end{proof}
     		
     		The separation system~${\vS}$ is \emph{separable} if for all non-trivial and non-degenerate ${\vrone}, {\rvtwo} \in {\vS}$ with ${\vrone} \leq {\vrtwo}$ there exists an ${\vs} \in {\vS}$ which emulates~${\vrone}$ in~${\vS}$ while simultaneously~${\sv}$ emulates~${\rvtwo}$ in~${\vS}$.
     		
     		Given some $\F \sub 2^{{\vU}}$, we say that~${\vs}$\emph{ emulates~${\vr}$ in~${\vS}$ for~$\F$} if~${\vs}$ emulates~${\vr}$ in~${\vS}$ and for every star $\sigma \sub {\vS} \setminus \{ {\rv} \}$ with $\sigma \in \F$ and every ${\vx} \in \sigma$ with ${\vx} \geq {\vr}$ we have $\sigma\!_{{\vx}}^{{\vs}} \in \F$.
     		
     		The separation system~${\vS}$ is \emph{$\F$-separable} if for all non-trivial and non-degener\-ate ${\vrone}, {\rvtwo} \in {\vS}$ with ${\vrone} \leq {\vrtwo}$ and $\{ {\rvone} \}, \{ {\vrtwo} \} \notin \F$ there exists an ${\vs} \in {\vS}$ which emulates~${\vrone}$ in~${\vS}$ for~$\F$ while simultaneously~${\sv}$ emulates~${\rvtwo}$ in~${\vS}$ for~$\F$.
	     	
	    	\begin{theorem}[{{\cite[Theorem~4.3]{TangleTreeAbstract}}}] \label{duality diestel oum}
	     		Let ${\vU}$ be a universe of separations and ${\vS} \sub {\vU}$ a separation system. Let $\F \sub 2^{{\vU}}$ be a set of stars, standard for~${\vS}$. If~${\vS}$ is $\F$-separable, then exactly one of the following holds:
	     		\begin{enumerate}[\rm (i)]\itemsep=0pt\vskip-\smallskipamount\vskip0pt
	     			\item There exists an $\F$-tangle of~$S$.
	     			\item There exists an $S$-tree over~$\F$.
	     		\end{enumerate}
	     	\end{theorem}
	     	
	  In applications of Theorem~\ref{duality diestel oum} it is often easier to split the proof of the main premise, that~${\vS}$ is $\F$-separable, into two parts: a~proof that~${\vS}$ is separable and one that $\F$ is \emph{closed under shifting} in~${\vS}$: that whenever ${\vs} \in {\vS}$ emulates (in~${\vS}$) some nontrivial and nondegenerate ${\vr} \leq {\vs}$ not forced by $\F$, then it does so for~$\F$. Indeed, the following is immediate from the definitions:
	     		
     		\begin{lemma}\label{sep and closed}
     			Let~${\vU}$ be a universe of separations, ${\vS} \sub {\vU}$ a separation system, and $\F \sub 2^{{\vU}}$ a set of stars. If~${\vS}$ is separable and~$\F$ is closed under shifting, then~${\vS}$ is $\F$-separable. \qed
     		\end{lemma}	
	     	
It is shown in~\cite{TangleTreeGraphsMatroids} that if ${\vU}$ is a universe of separations with an order function, then the sets ${\vSk}$ of all separations of order less than some fixed positive integer~$k$ are separable for all~$k$, and virtually all the applications of Theorem~\ref{duality diestel oum} that are given in~\cite{TangleTreeAbstract} involve a separation system of the form $\vSk$.
	      
	       While many applications of the submodularity of an order function use only its structural consequence that motivated our abstract notion of submodularity, the use of submodularity in the proof that~$\vSk$ is separable~-- see~\cite[Lemma~3.4]{TangleTreeGraphsMatroids}~-- uses it in a more subtle way. There, the orders of opposite corners of two crossing separations $\vs$ and~$\vt$ are compared not with any fixed value of~$k$ but with the (possibly distinct) orders of~$s$ and~$t$ directly. This kind of argument is naturally difficult, if not impossible, to mimic in our set-up.%
   \COMMENT{}

	However, we can prove this nevertheless, choosing a different route. The following lemma is, in essence, the main result of this section: 
	     
	     \begin{lemma}\label{submod sep sys}
	     Let ${\vU}$ be a universe of separations and ${\vS} \sub {\vU}$ a submodular separation system. Then ${\vS}$ is separable.
	     \end{lemma}
	     	
We will actually prove a slightly more general statement about submodular lattices. Let $(L, \vee, \wedge)$ be a lattice and let $M \sub L$. Given $x, y \in M$, we say that~$x$ \emph{pushes}~$y$ if $x \leq y$ and for any $z \in M$ with $z \leq y$ we have $x \wedge z \in M$. Similarly, we say that~$x$ \emph{lifts}~$y$ if $x \geq y$ and for any $z \in M$ with $z \geq y$ we have $x \vee z \in M$. Observe that both of these relations are reflexive and transitive: Every $x \in M$ pushes (lifts) itself and if~$x$ pushes (lifts)~$y$ and~$y$ pushes (lifts)~$z$, then~$x$ pushes (lifts)~$z$. We say that~$M$ is \emph{strongly separable} if for all $x, y \in M$ with $x \leq y$ there exists a $z \in M$ that lifts~$x$ and pushes~$y$. 
	    
The definitions of lifting, pushing and separable extend verbatim to a separation system within a universe of separations when regarded as a subset of the underlying lattice. These notions are strengthenings of the notions of emulating and separable: If ${\vs} \in {\vS}$ lifts ${\vr} \in {\vS}$, then~${\vs}$ emulates~${\vr}$ in~${\vS}$, and~${\vs}$ pushes~${\vr}$ if and only if~${\sv}$ lifts~${\rv}$. Similarly, if~${\vS}$ is strongly separable, then~${\vS}$ is separable. Lemma \ref{submod sep sys} is then an immediate consequence of the following:
	
\begin{lemma} \label{submod lattice}
Let~$L$ be a finite lattice and $M \sub L$ submodular. Then~$M$ is strongly separable.
\end{lemma}
     		
     		\begin{proof}
     		Call a pair $(a, b) \in M \times M$ \emph{bad} if $a \leq b$ and there is no $x \in M$ that lifts~$a$ and pushes~$b$. Assume for a contradiction that there was a bad pair and choose one, say $(a, b)$, such that $I(a, b) := \{ u \in M \colon a \leq u \leq b \}$ is minimal.
     		
     		We claim that~$a$ pushes every $z \in I(a,b) \setminus \{ b \}$. Indeed, assume for a contradiction~$a$ did not push some such~$z$. By minimality of $(a,b)$, the pair $(a, z)$ is not bad, so there is some $x \in M$ which lifts~$a$ and pushes~$z$. By assumption, $x \neq a$ and so by minimality, the pair $(x, b)$ is not bad, yielding a $y \in M$ which lifts~$x$ and pushes~$b$. By transitivity, it follows that~$y$ lifts~$a$. But then $(a,b)$ is not a bad pair, which is a contradiction. An analogous argument establishes that~$b$ lifts every $z \in I(a,b) \setminus \{ a \}$.
     		
     		Since $(a, b)$ is bad, $a$ does not push~$b$, so there is some $x \in M$ with $x \leq b$ for which $a \wedge x \notin M$. Similarly, there is a $y \in M$ with $y \geq a$ for which $b \vee y \notin M$. Since~$M$ is submodular, it follows that $a \vee x, b \wedge y \in M$. Note that $a \vee x, b \wedge y \in I(a,b)$. Furthermore, $x \leq a \vee x$ and $a \wedge x \notin M$, so~$a$ does not push $a \vee x$. We showed that~$a$ pushes every $z \in I(a,b) \setminus \{ b \}$, so it follows that $a \vee x = b$. Similarly, we find that $b \wedge y = a$. But then
     		\begin{align*}
     			x \vee y &= x \vee (a \vee y) = 	b \vee y \notin M , \\
     			x \wedge y &= (x \wedge b) \wedge y = x \wedge a \notin M .
     		\end{align*}
     		This contradicts the submodularity of~$M$.
     		\end{proof}	
	
As a result we obtain our tangle-tree duality theorem for $\F$-tangles of submodular separation systems, which we restate:
	
	\begin{Fduality}\label{submodular duality}
	Let ${\vU}$ be a universe of separations and ${\vS} \sub {\vU}$ a submodular separation system. Let $\F \sub 2^{{\vU}}$ be a set of stars which is standard for~${\vS}$ and closed under shifting. Then exactly one of the following holds:
	     		\begin{enumerate}[\rm (i)]\itemsep=0pt\vskip-\smallskipamount\vskip0pt
	     			\item There exists an $\F$-tangle of~$S$.
	     			\item There exists an $S$-tree over~$\F$.
	     		\end{enumerate}
	     	\end{Fduality}
\begin{proof}
Since~${\vS}$ is submodular, Lemma~\ref{submod sep sys} implies that~${\vS}$ is separable. Since~$\F$ is closed under shifting, it follows from Lemma~\ref{sep and closed} that~${\vS}$ is $\F$-separable. The result then follows by Theorem~\ref{duality diestel oum}.
\end{proof}
	
We will now use Theorem~\ref{submodular duality} to prove Theorem~\ref{TTDThm2}, which in turn  implies Theorem~\ref{TTDThm}. Recall that we are considering a downclosed subset $Q \sub {\vU}$ of a distributive universe of separations, and a submodular separation system ${\vS}$ without degenerate elements in $\subseteq {\vU}$, and we wish to prove a tangle-tree duality theorem for abstract tangles of ${\vS}$ extending $Q$. Note that these are precisely the $(\T \cup \F_Q)$-tangles of ${\vS}$. However, since the family~$\F$ in Theorem~\ref{submodular duality} is assumed to be a set of stars, we cannot work directly with $\T$. Instead we will work with $\T^*$, the set of stars in $\T$. It will turn out that, since~${\vS}$ has no degenerate elements, this will not change the set of $\T$-tangles (cf.\ Lemma~\ref{tangle stars}). So, we will first show that we can apply Theorem~\ref{submodular duality} with $\F = \T_Q$, where $\T_Q := \T^* \cup \F_Q$, and then show that the $\T^*$-tangles are precisely the $\T$-tangles. Theorem~\ref{TTDThm2} will then follow.
	     	
Let us first prove the following simple fact, which will be useful in a few different situations.
	     	
	     	   	\begin{lemma} \label{push co-small}
		Let~${\vU}$ be a distributive universe of separations. Let ${\vu}, {\vv}, {\vw} \in {\vU}$. If ${\vu} \leq {\vv}$ and ${\vv} \vee {\vw}$ is co-small, then ${\vv} \vee ({\vw} \wedge {\uv})$ is co-small.
	\end{lemma}

\proof
		Let ${\vx} := {\vv} \vee ({\vw} \wedge {\uv})$. By distributivity of~${\vU}$
		\[
		{\vx} = ( {\vv} \vee {\vw} ) \wedge ({\vv} \vee {\uv} ) \geq ({\vv} \vee {\vw}) \wedge ({\vu} \vee {\uv} ).
		\]
		Let ${\vs} := {\vv} \vee {\vw}$ and ${\vt} := {\vu} \vee {\uv}$. Then ${\sv} \leq {\vs}$ by assumption and \mbox{${\sv} \leq {\vleft} \leq {\vt}$}. Further ${\tv} \leq {\vu} \leq {\vt}$ and ${\tv} \leq {\vu} \leq {\vv}$. Therefore
		\[
			{\xv} \leq {\sv} \vee {\tv} \leq {\vs} \wedge {\vt} \leq {\vx} .\eqno\qed
		\]

		In order to apply Theorem \ref{submodular duality} with $\F = \T_Q$, we need to show that $T_Q$ is closed under shifting.
		
		\begin{lemma} \label{closed under shifting}
     			If $Q \sub {\vU}$ is down-closed and~${\vU}$ is distributive, then~$\T_Q$ is closed under shifting.
     		\end{lemma}
     		
     		\begin{proof}
     			Let ${\vr} \in {\vS}$ non-trivial and non-degenerate with $\{ {\rv} \} \notin \F$. Let ${\vs} \in {\vS}$ emulate~${\vr}$ in~${\vS}$, let $\T_Q \ni \sigma \sub {\vS} \setminus \{ {\rv} \}$ and ${\vr} \leq {\vx} \in \sigma$. We have to show that $\sigma\!_{{\vx}}^{{\vs}} \in \T_Q$. From Lemma~\ref{shift is star} we know that $\sigma\!_{{\vx}}^{{\vs}}$ is a star, so we only need to verify that $\sigma\!_{{\vx}}^{{\vs}} \in \T^* \cup \F_Q$.

     			Suppose first that $\sigma \in \T^*$. Let ${\vw} := \bigvee (\sigma \setminus \{ {\vx} \})$. Applying Lemma~\ref{push co-small} with ${\vu} = {\vs}$ and ${\vv} = {\vx} \vee {\vs}$, we see that
     			\[
     				\bigvee \sigma\!_{{\vx}}^{{\vs}} = ({\vx} \vee {\vs}) \vee ({\vw} \wedge {\sv} )
     			\]
     			is co-small. Since~$\sigma\!_{{\vx}}^{{\vs}}$ has at most three elements, it follows that $\sigma\!_{{\vx}}^{{\vs}} \in \T$.

\goodbreak
     			
     			Suppose now that $\sigma \in \F_Q$. Then $\sigma = \{ {\vx} \}$ and ${\xv} \in Q$. As~$Q$ is down-closed, we have ${\xv} \wedge {\sv} \in Q$. Since~$\sigma\!_{{\vx}}^{{\vs}}$ is a star, ${\xv} \wedge {\sv}$ is non-degenerate and therefore
     			\[
     			\sigma\!_{{\vx}}^{{\vs}} = \{ {\vx} \vee {\vs} \} = \{ ({\xv} \wedge {\sv})^* \} \in \F_Q .
     			\]
     			
	\end{proof}		
	     	
	     	\begin{lemma} \label{tangle stars}
				Let~${\vU}$ be a distributive universe of separations and let ${\vS} \sub {\vU}$ be a submodular separation system without degenerate elements. Then the $\T^*$-tangles are precisely the abstract tangles. 
     		\end{lemma}
     		
     		\begin{proof}     			
     			Since $\T^* \sub \T$, every abstract tangle is also a $\T^*$-tangle. We only need to show that, conversely, every $\T^*$-tangle in fact avoids~$\T$. 
     			
					For $\sigma \in \T$, let $d( \sigma )$ be the number of pairs ${\vs}, {\vt} \in \sigma$ which are not nested. Let~$O$ be a consistent orientation of~$S$ and suppose~$O$ was not an abstract tangle. Choose $\T \ni \sigma \sub O$ such that $d( \sigma)$ is minimum and, subject to this, $\sigma$ is inclusion-minimal. We will show that~$\sigma$ is indeed a star, thus showing that~$O$ is not a $\T^*$-tangle.
					
					If~$\sigma$ contained two comparable elements, say ${\vs} \leq {\vt}$, then $\sigma ' := \sigma \setminus \{ {\vs} \}$ satisfies $\sigma ' \in \T$, $\sigma ' \sub O$ and $d( \sigma') \leq d( \sigma)$, violating the fact that~$\sigma$ is inclusion-minimal. Hence~$\sigma$ is an antichain. Since~${\vS}$ has no degenerate elements, it follows from the consistency of~$O$ that any two nested ${\vs}, {\vt} \in \sigma$ satisfy ${\vs} \leq {\tv}$. To show that~$\sigma$ is a star, it thus suffices to prove that any two elements are nested. 
					
					Suppose that~$\sigma$ contained two crossing separations, say ${\vs}, {\vt} \in \sigma$. By submodularity of~${\vS}$, at least one of ${\vs} \wedge {\tv}$ and ${\sv} \wedge {\vt}$ lies in~${\vS}$. By symmetry we may assume that ${\vr} := {\vs} \wedge {\tv} \in {\vS}$. Let $\sigma ' := ( \sigma \setminus \{ {\vs} \} ) \cup \{ {\vr} \}$. Since~$O$ is consistent, ${\vr} \leq {\vs}$ and $r \neq s$, it follows that ${\vr} \in O$ and so $\sigma ' \sub O$ as well. 
					Let ${\vw} = \bigvee ( \sigma \setminus \{ {\vt} \} )$. As ${\vt} \vee {\vw} = \bigvee \sigma$ is co-small, we can apply Lemma~\ref{push co-small} with ${\vu} = {\vv} = {\vt}$ to deduce that ${\vt} \vee ({\vw} \wedge {\tv})$ is co-small as well. But
					\[
					{\vt} \vee ({\vw} \wedge {\tv}) = {\vt} \vee \bigvee_{{\vx} \in \sigma \setminus \{ {\vt} \}} ({\vx} \wedge {\tv}) \leq \bigvee \sigma ' ,
					\]
					so $\bigvee \sigma '$ is also co-small and $\sigma ' \in \T$.
					
					We now show that $d( \sigma ') < d( \sigma)$. Since~$s$ and~$t$ cross, while~$r$ and~$t$ do not, it suffices to show that every ${\vx} \in \sigma \setminus \{ {\vs} \}$ which is nested with~${\vs}$ is also nested with~${\vr}$. But for every such~${\vx}$ we have ${\vs} \leq {\xv}$. Since ${\vr} \leq {\vs}$, we get ${\vr} \leq {\xv}$ as well, showing that~$r$ and~$x$ are nested. So in fact $d( \sigma ') < d(\sigma)$, which is a contradiction. This completes the proof that~$\sigma$ is nested and therefore a star.
     		\end{proof}
	     	
We are now in a position to prove Theorem~\ref{TTDThm2}.
	     	
	     	\begin{proof}[Proof of Theorem~\ref{TTDThm2}]
	     	By Lemma \ref{submod sep sys}, ${\vS}$ is separable, and by Lemma \ref{closed under shifting}, $\T_Q$ is closed under shifting. Therefore, by Theorem \ref{submodular duality}, there is no $S$-tree over $\T^* \cup \F_Q$, if and only if $S$ has a $T_Q$-tangle, that is, a $\T^*$-tangle extending $Q$.
	     	
	     	However, since ${\vU}$ is distributive and ${\vS}$ contains no degenerate elements, Lemma \ref{tangle stars} implies that $S$ has a $\T^*$-tangle extending $Q$ if and only if $S$ has an abstract tangle extending $Q$.
	     	\end{proof}
	      
     \end{section}

     \begin{section}{Special cases and applications}\label{apps}
     
     	\begin{subsection}{Tangles in graphs and matroids}
     			     			
     			We briefly indicate how tangles in graphs and matroids can be seen as special cases of abstract tangles in separation systems. Tangles in graphs and hypergraphs were introduced by Robertson and Seymour in~\cite{GMX}, but a good deal of the work is done in the setting of connectivity systems. Geelen, Gerards and Whittle~\cite{BranchDecMatroids} made this more explicit and defined tangles as well as the dual notion of branch-decompositions for connectivity systems, an approach that we will follow.
     			
     			Let~$X$ be a finite set and $\lambda : 2^X \to \mathbb{Z}$ a map assigning integers to the subsets of~$X$ such that $\lambda(X \setminus A) = \lambda(A)$ for all $A \sub X$ and 
     			\[
     			\lambda( A \cup B) + \lambda(A \cap B) \leq \lambda(A) + \lambda(B)
     			\]
     			 for all $A,B \sub X$. The pair $(X, \lambda)$ is then called a \emph{connectivity system}.
     			
     			Both graphs and matroids give rise to connectivity systems. For a given graph~$G$, we can take $X := E(G)$ and define $\lambda(F)$ as the number of vertices of~$G$ incident with edges in both~$F$ and $E \setminus F$. Given a matroid~$M$ with ground-set~$X$ and rank-function~$r$, we take~$\lambda$ to be the connectivity function $\lambda(A) := r(A) + r(X \setminus A) - r(X)$.
     			
     			Now consider~$2^X$ as a universe of separations with set-inclusion as the partial order and $A^* = X \setminus A$ as involution. For an integer~$k$, the set~${\vSk}$ of all sets~$A$ with $\lambda(A) < k$ is then a submodular separation system. Let $Q := \{ \emptyset \} \cup \{ \{ x \} \colon x \in X \}$ consist of the empty-set and all singletons of~$X$ and note that~$Q$ is down-closed.

     			A \emph{tangle of order~$k$} of $(X, \lambda)$, as defined in~\cite{BranchDecMatroids}, is then precisely an abstract tangle extending~$Q$. It is easy to see that $(X, \lambda)$ has a \emph{branch-decomposition} of width~$<\! k$ if and only if there exists an $S_k$-tree over $\T^* \cup \F_Q$. Theorem~\ref{TTDThm2} then yields the classic duality theorem for tangles and branch-decompositions in connectivity systems, see~\cite{GMX, BranchDecMatroids}.
     			
     	\end{subsection}

     	\begin{subsection}{Clique separations}\label{sec:cliques}
     
     		We now describe a submodular separation system that is not derived from a submodular order function, and provide a natural set of stars for which Theorem~\ref{submodular duality} applies.
     		
     		Let $G = (V, E)$ be a finite graph and~${\vU}$ the universe of all separations of~$G$, that is, pairs $(A, B)$ of subsets of~$V$ with $V = A \cup B$ such that there is no edge between $A \setminus B$ and $B \setminus A$. Here the partial order is given by $(A, B) \leq (C,D)$ if and only if $A \sub C$ and $B \supseteq D$, and the involution is simply $(A,B)^* = (B,A)$. For $(A, B) \in {\vU}$, we call $A \cap B$ the \emph{separator} of $(A, B)$. It is an \emph{$a$-$b$-separator} if $a \in A \setminus B$ and $b \in B \setminus A$. We call $A \cap B$ a \emph{minimal separator} if there exist $a \in A \setminus B$ and $b \in B \setminus A$ for which $A \cap B$ is an inclusion-minimal $a$-$b$-separator.
     	
     	Recall that a \emph{hole} in a graph is an induced cycle on more than three vertices. A graph is \emph{chordal} if it has no holes.

     \begin{theorem}[Dirac~\cite{dirac61}] \label{dirac chordal}
     	A graph is chordal if and only if every minimal separator is a clique.
     \end{theorem}
     		
     		 Let~${\vS}$ be the set of all $(A, B) \in {\vU}$ for which $G[A \cap B]$ is a clique. We call these the \emph{clique separations}. Note that~${\vS}$ is closed under involution and therefore a separation system. To avoid trivialities, we will assume that the graph~$G$ is not itself a clique. In particular, this implies that~${\vS}$ contains no degenerate elements.
     		
     		\begin{lemma} \label{clique sep submodular}
     			Let $s, t \in S$. At least three of the four corners of~$s$ and~$t$ are again in~${\vS}$. In particular, ${\vS}$ is submodular.
     		\end{lemma}
     		
     		\begin{proof}
     			Let ${\vs} = (A,B)$ and ${\vt} = (C,D)$. Since $G[A \cap B]$ is a clique and $(C, D)$ is a separation, we must have $A \cap B \sub C$ or $A \cap B \sub D$, without loss of generality $A \cap B \sub C$. Similarly, it follows that $C \cap D \sub A$ or $C \cap D \sub B$; we assume the former holds. For each corner other than ${\vs} \wedge {\vt} = (A \cap C, B \cup D)$, the separator is a subset of either $A \cap B$ or $C \cap D$ and therefore the subgraph it induces is a clique. This proves our claim.	     			
     		\end{proof}

     Suppose that the graph~$G$ contains a hole~$H$. Then for every $(A, B) \in {\vS}$,  either $H \sub A$ or $H \sub B$. In this way, every hole~$H$ induces an orientation 
     \[
     O_H := \{ (A,B) \in {\vS} \colon H \sub B \}
     \]
      of~${\vS}$. We now describe these orientations as tangles over a suitable set of stars.
     
     Let $\F \sub 2^{{\vU}}$ be the set of all sets $\{ (A_1, B_1), \ldots (A_n, B_n) \} \sub {\vU}$ for which $G[\bigcap B_i]$ is a clique (note that the graph without any vertices is a clique). As usual, we denote by~$\F^*$ the set of all elements of~$\F$ which are stars.
     
     \goodbreak
     
     \begin{theorem} \label{holes as tangles}
     	Let~$O$ be an orientation of~${S}$. Then the following are equivalent:
     	\begin{enumerate}[\rm (i)]\itemsep=0pt\vskip-\smallskipamount\vskip0pt
     		\item $O$ is an $\F^*$-tangle.
     		\item $O$ is an $\F$-tangle.
     		\item There exists a hole~$H$ with $O = O_H$.
     	\end{enumerate}
     \end{theorem}
     
	It is easy to see that every orientation~$O_H$ induced by a hole~$H$ is an \mbox{$\F$-tangle}. To prove that, conversely, every \mbox{$\F$-tangle} is induced by a hole, we use Theorem~\ref{dirac chordal} and an easy observation about clique-separators, Lemma~\ref{clique separator transitive} below. The proof that every \mbox{$\F^*$-tangle} is already an \mbox{$\F$-tangle}, the main content of Lemma~\ref{clique sep tangle star} below, is similar to the proof of Lemma~\ref{tangle stars}, but some care is needed to keep track of the separators of two crossing separations.
          
     For a set $\tau \sub {\vU}$, let $J(\tau) := \bigcap_{(A,B) \in \tau} B$ be the intersection of all the right sides of separations in~$\tau$, where $J(\emptyset) := V(G)$. 
     
     \begin{lemma} \label{clique separator transitive}
     	Let~$\tau$ be a set of clique separations, $J = J(\tau)$ and $K \sub J$. Let $a, b \in J \setminus K$. If~$K$ separates~$a$ and~$b$ in $G[J]$, then it separates them in~$G$.
     \end{lemma}
     
     \begin{proof}
     We prove this by induction on~$|\tau|$, the case $\tau = \emptyset$ being trivial. Suppose now $|\tau| \geq 1$ and let $(X,Y) \in \tau$ arbitrary. Put $\tau' := \tau \setminus \{ (X,Y) \}$ and $J' := J(\tau')$. Note that $J = J' \cap Y$. Let $G' := G[J']$ and $(X', Y') := (X \cap J', Y \cap J')$. 
     
     Then $K \sub J'$ and $a, b \in J' \setminus K$. Suppose~$K$ did not separate~$a$ and~$b$ in~$G'$ and let $P \sub J'$ be an induced $a$-$b$-path avoiding~$K$. Since $G'[ X' \cap Y']$ is a clique, $P$ has at most two vertices in $X' \cap Y'$ and they are consecutive vertices along~$P$. As $a, b \in Y'$ and $(X',Y')$ is a separation of~$G'$, it follows that $P \sub Y'$. But then~$K$ does not separate~$a$ and~$b$ in $J = J' \cap Y$, contrary to our assumption.
     
     Hence~$K$ separates~$a$ and~$b$ in~$G'$. By inductive hypothesis applied to~$\tau'$, it follows that~$K$ separates~$a$ and~$b$ in~$G$.
     \end{proof}
		
		\begin{lemma} \label{clique sep tangle star}
			Every $\F^*$-tangle is an $\F$-tangle and a regular profile.
		\end{lemma}     
		
		\begin{proof}
			Let~$P$ be an $\F^*$-tangle. It is clear that~$P$ contains no co-small separation, since $\{ (V,A) \} \in \F^*$ for every co-small $(V,A) \in \vS$. Since~$P$ is consistent, it follows that~$P$ is in fact down-closed. 
			
			We now show that~$P$ is a profile. Let $(A,B), (C,D) \in P$ and assume for a contradiction that $(E,F) := ((A,B) \vee (C,D))^* \in P$. Recall that either $C \cap D \subseteq A$ or $C \cap D \subseteq B$. 
	
			\begin{figure}[ht!]
				\begin{center}
				\resizebox{4cm}{!}{	\begin{tikzpicture}

\draw[pattern=north west lines, pattern color=gray] (0,0) ellipse (5cm and 1cm);
\fill[white] (0,0) ellipse (1cm and 5cm);
\fill[white] (0,-5) rectangle (5,5);
\draw[thick, green] (0,0) [partial ellipse=279.5:360:4.8cm and .8cm];
\draw[thick, green] (0,0) [partial ellipse=0:99.5:4.8cm and .8cm];
\draw[thick, green] (0,0) [partial ellipse=170.5:350.5:.8cm and 4.8cm];
\draw[red,thick] (0,0) ellipse (5cm and 1cm);
\draw[blue,thick] (0,0) ellipse (1cm and 5cm);
\node[text= red, scale=2] at (5.5,1) {$C$};
\node[text= red, scale=2]  at (5.5,-1) {$D$};
\node[text= blue, scale=2]  at (1,5.5) {$B$};
\node[text= blue, scale=2]  at (-1,5.5) {$A$};
\node[text= green, scale=2]  at (1,-5.5) {$E$};
\node[text= green, scale=2]  at (-1,-5.5) {$F$};
\node[scale=2] at (-2.5,0) {$\emptyset$};

\end{tikzpicture}}
				\end{center}\vskip-12pt
				\caption{The case $C \cap D \sub B$}
				\label{fig: hole profile}
			\end{figure}

			Suppose first that $C \cap D \sub B$; this case is depicted in Figure~\ref{fig: hole profile}. Let $(X,Y) := (A,B) \wedge (D,C)$ and note that $X \cap Y \sub A \cap B$, so that $(X,Y) \in {\vS}$. It follows from the consistency of~$P$ that $(X,Y) \in P$. Let $\tau := \{ (C,D), (E,F), (X,Y) \}$ and observe that $\tau \sub P$ is a star. However 	
			\[
			J( \tau ) = D  \cap (A \cup C) \cap (B \cup C) = (D \cap B) \cap (A \cup C) ,
			\]
			which is the separator of $(E,F)$. Since $(E,F) \in {\vS}$, $G[J(\tau)]$ is a clique, thereby contradicting the fact that~$P$ is an $\F^*$-tangle.
			
			Suppose now that $C \cap D \sub A$. Let $(X,Y) := (B,A) \wedge (C,D)$ and note that $X \cap Y \sub A \cap B$, so that $(X,Y) \in {\vS}$. Since~$P$ is down-closed, it follows that $(X,Y) \in P$. Therefore $\tau := \{ (A,B), (E,F), (X,Y) \} \sub P$. But~$\tau$ is a star and
			\[
			J( \tau ) = B  \cap (A \cup C) \cap (A \cup D) = B \cap (A \cup (C \cap D) )= B \cap A,
			\]
			and so $G[J(\tau)]$ is a clique, which again contradicts our assumption that~$P$ is an $\F^*$-tangle. This contradiction shows that~$P$ is indeed a profile.

	We now prove that for any $\tau \sub P$ there exists a star $\sigma \sub P$ with $J(\sigma) = J(\tau)$. It follows then, in particular, that~$P$ is an $\F$-tangle.
			
			Given $\tau \sub P$, choose $\sigma \sub P$ with $J(\sigma) = J(\tau)$ so that $d(\sigma)$, the number of crossing pairs of elements of~$\sigma$, is minimum and, subject to this, $\sigma$ is inclusion-minimal. Then~$\sigma$ is an antichain: If $(A,B) \leq (C,D)$ and both $(A,B), (C,D) \in \sigma$, then $\sigma ' := \sigma \setminus \{ (A,B) \}$ satisfies $J(\sigma') = J(\sigma)$, thus violating the minimality of~$\sigma$. Since $\sigma \sub P$ and~$P$ is consistent, no two elements of~$\sigma$ point away from each other. Therefore, any two nested elements of~$\sigma$ point towards each other. To verify that~$\sigma$ is a star, it suffices to check that~$\sigma$ is nested. 
			
			Assume for a contradiction that~$\sigma$ contained two crossing separations $(A,B)$ and $(C,D)$. If $(E,F) := (A,B) \vee (C,D) \in {\vS}$, obtain~$\sigma'$ from~$\sigma$ by deleting $(A,B)$ and $(C,D)$ and adding $(E,F)$. We have seen above that~$P$ is a profile, so $\sigma ' \sub P$. By Lemma~\ref{fish lemma}, every element of $\sigma \setminus \{ (A,B), (C,D) \}$ that is nested with both $(A,B)$ and $(C,D)$ is also nested with $(E,F)$. Since~$\sigma'$ misses the crossing pair $\{ (A,B), (C,D) \}$, it follows that $d(\sigma ') < d(\sigma)$. But $J(\sigma') = J(\sigma)$, contradicting the minimality of~$\sigma$.
			
			Hence it must be that $(E,F) \notin {\vS}$, so $A \cap B \not \sub C$ and $C \cap D \not \sub A$. Therefore $(X,Y) := (A,B) \wedge (D,C) \in {\vS}$. Let $\sigma ' := (\sigma \setminus \{ (A,B) \} ) \cup \{ (X,Y) \}$. Note that $(X,Y) \leq (A,B) \in P$, so $\sigma ' \sub P$. Moreover $Y \cap D = (B \cup C) \cap D = B \cap D$, since $C \cap D \sub B$. Therefore $J(\sigma ') = J(\sigma)$. As mentioned above, any $(U,W) \in \sigma \setminus \{ (A,B) \}$ that is nested with $(A,B)$ satisfies $(A,B) \leq (W,U)$. Therefore $(X,Y) \leq (A,B) \leq (W,U)$, so $(X,Y)$ is also nested with $(U,W)$. It follows that $d( \sigma ') < d(\sigma)$, which is a contradiction. This completes the proof that~$\sigma$ is nested and therefore a star.	
		\end{proof}
     
     \goodbreak
     
     	\begin{proof}[Proof of Theorem~\ref{holes as tangles}]
		(i) $\rightarrow$ (ii): See Lemma~\ref{clique sep tangle star}.
		
		(ii) $\rightarrow$ (iii): Let~$O$ be an $\F$-tangle and $J := J(O)$. We claim that there is a hole~$H$ of~$G$ with $H \sub J$. Such a hole then trivially satisfies $O_H = O$.
		
		Assume there was no such hole, so that $G[J]$ is a chordal graph. Since~$O$ is $\F$-avoiding, $G[J]$ itself cannot be a clique, so there exists a minimal set $K \sub J$ separating two vertices $a, b \in J \setminus K$ in $G[J]$. By Theorem~\ref{dirac chordal}, $K$ induces a clique in~$G$. By Lemma~\ref{clique separator transitive}, $K$ separates~$a$ and~$b$ in~$G$, so there exists a separation $(A, B) \in {\vS}$ with $A \cap B = K$, $a \in A \setminus B$ and $b \in B \setminus A$. As~$O$ orients~${\vS}$, it must contain one of $(A,B), (B,A)$, say without loss of generality $(A,B) \in O$. But then $J \sub B$, contrary to $a \in J$. This proves our claim.
		
		(iii) $\rightarrow$ (i): We have $H \sub J(O_H)$, so $J(O)$ does not induce a clique. Since every star $\sigma \subseteq O$ has $J(O) \subseteq J(\sigma)$ there is no star $\sigma \subseteq O$ such that $G[J(\sigma)]$ is a clique, and so $O$ is $\mathcal{F}^*$-avoiding. Furthermore, $O_H$ is clearly consistent, and so $O$ is an $\mathcal{F}^*$-tangle.
\end{proof}

	The upshot of Theorem~\ref{holes as tangles} is that a hole in a graph, although a very concrete substructure, can be regarded as a tangle. This is in line with our general narrative, set forth e.g.\ in~\cite{TangleTreeGraphsMatroids,TanglesEmpirical,TanglesSocial}, that tangles arise naturally in very different contexts, and underlines the expressive strength of abstract separation systems and tangles. 
	
	What does our abstract theory then tell us about the holes in a graph? The results we will derive are well-known and not particularly deep, but it is none\-theless remarkable that the theory of abstract separation systems, emanating from the theory of highly connected substructures of a graph or matroid, is able to express such natural facts about holes.
     
	Firstly, by Lemma~\ref{clique sep tangle star}, every hole induces a profile of~${S}$. Hence Theorem~\ref{tangle tree reg prof} applies and yields a nested set~$\N$ of clique-separations distinguishing all holes which can be separated by a clique. This is similar to, but not the same as, the decomposition by clique separators of Tarjan~\cite{tarjan1985}: the algorithm in~\cite{tarjan1985} essentially produces a maximal nested set of clique separations and leaves `atoms' that do not have any clique separations, whereas our tree set merely distinguishes the holes and leaves larger pieces that might allow further decomposition.
	
	Secondly, we can apply Theorem~\ref{submodular duality} to find the structure dual to the existence of holes. It is clear that~$\F^*$ is standard, since~$\F^*$ contains $\{ (V,A) \}$ for every $(V,A) \in {\vS}$.

     \begin{lemma} \label{clique sep closed}
     $\F^*$ is closed under shifting.
     \end{lemma}
     
     \begin{proof}
     Let $(X,Y) \in {\vS}$ emulate a non-trivial $(U,W) \in {\vS}$ with $\{ (W,U) \} \notin \F^*$, let $\sigma = \{ (A_i, B_i) \colon 0 \leq i \leq n \} \sub {\vS}$ with $\sigma \in \F^*$ and $(U,W) \leq (A_0, B_0)$. Then
          \[
   \sigma ' :=  \sigma_{(A_0,B_0)}^{(X,Y)} = \{ (A_0 \cup X, B_0 \cap Y) \} \cup \{ (A_i \cap Y, B_i \cup X) \colon 1 \leq i \leq n \} .
     \]
     By Lemma~\ref{shift is star}, $\sigma' \sub {\vS}$ is a star. We need to show that $G[J(\sigma')]$ is a clique.%
       \COMMENT{}
     
     Let $(A,B) := \bigvee_{i \geq 1} (A_i, B_i)$ and note that $(A,B) \leq (B_0, A_0)$, since~$\sigma$ is a star. Then
     \[
     (B,A) \wedge (V, B_0) = (B, B_0) \in {\vU} .
     \]
     But $G[B \cap B_0] = G[J(\sigma)]$ is a clique, so in fact $(B, B_0) \in {\vS}$. Since $(U,W) \leq (A_0,B_0) \leq (B,A)$, we see that $(U,W) \leq (B,B_0)$. As $(X,Y)$ emulates $(U,W)$ in~${\vS}$, we find that $(E,F) := (X,Y) \vee (B,B_0) \in {\vS}$. It thus follows that
     \[
     J( \sigma ' ) = (X \cup B) \cap (Y \cap B_0) = E \cap F
     \]
    and so $G[J(\sigma')]$ is indeed a clique. Therefore $\sigma ' \in \F^*$.
     \end{proof}
     
	\begin{theorem} \label{duality for holes}
		Let~$G$ be a graph. Then the following are equivalent:
			\begin{enumerate}[\rm (i)]\itemsep=0pt\vskip-\smallskipamount\vskip0pt
				\item $G$ has a tree-decomposition in which every part is a clique.
				\item There exists an ${S}$-tree over~$\F^*$.
				\item $S$ has no $\F^*$-tangle.
				\item $G$ is chordal.
			\end{enumerate}
	\end{theorem}
	
	\begin{proof}
	(i) $\rightarrow$ (ii): Let $(T, \V)$ be a tree-decomposition of~$G$ in which every part is a clique. For adjacent $s, t \in T$, let $T_{s,t}$ be the component of $T - st$ containing~$t$ and let~$V_{s,t}$ be the union of all~$V_u$ with $u \in T_{s,t}$. Define $\alpha : {\vec E}(T) \to {\vU}$ as $\alpha(s,t) := (V_{t,s}, V_{s,t})$. Then $\alpha(s,t) = \alpha(t,s)^*$. The separator of $\alpha(s,t)$ is $V_s \cap V_t$, which is a clique by assumption. Hence $(T, \alpha)$ is in fact an $S$-tree. It is easy to see that $\alpha(F_t)$ is a star for every $t \in T$ and that $J(\alpha(F_t))= V_t$. Therefore $(T, \alpha)$ is an $S$-tree over~$\F^*$.
	
	(ii) $\rightarrow$ (i): Given an $S$-tree $(T, \alpha)$ over~$\F^*$, define $V_t := J( \alpha(F_t))$ for $t \in T$. It is easily verified that $(T, \V)$ is a tree-decomposition of~$G$. Each~$V_t$ is then a clique, since $\alpha(F_t) \in \F$.
	
	(ii) $\leftrightarrow$ (iii): Follows from Theorem~\ref{submodular duality}, since~$\F^*$ is standard for~${\vS}$ and closed under shifting by Lemma~\ref{clique sep closed}.
	
	(iii) $\leftrightarrow$ (iv): Follows from Theorem~\ref{holes as tangles}.
	\end{proof}
     
     The equivalence of~(i) and~(iv) is a well-known characterization of chordal graphs that goes back to a theorem Gavril~\cite{gavril74} which identifies chordal graphs as the intersection graphs of subtrees of a tree.

		\end{subsection}     

     	\begin{subsection}{Tangle-tree duality in cluster analysis}\label{sec:cluster}

Let us now apply Theorem~\ref{submodular duality} to a generic scenario in cluster analysis~\cite{TanglesEmpirical,TanglesSocial}, where $V$ is thought of as a data set, $S$~is a set of certain `natural' bipartitions of~$V\!$, and we are interested in certain $\F$-tangles as `clusters'. The idea is that clusters should be described by these $\F$-tangles in the same way as the vertex set of a large grid in a graph is captured by the graph tangle~$\tau$ it induces: although every oriented separation $\vs$ in~$\tau$ points to most of the vertices of the grid, the cluster can be `fuzzy' in that these are not the same points for every $\vs\in\tau$. Indeed, there need not be a single vertex to which all the $\vs\in\tau$ point.

To mimic this idea, we want to choose $\F$ so that, whenever we consider just a few separations in~$S$, any $\F$-tangle $\tau$ of~$S$ must orient these so that they all point to at least some~$m$ (say) points in~$V\!$, while we do not require that the intersection of all the sets $B$ for $(A,B)\in\tau$ must be large (or even non-empty).

Formally, then, let $\vU$ be the universe of all oriented bipartitions~$(A,B)$ of some non-empty%
   \COMMENT{}
   set~$V\!$, including $(\emptyset,V)$ and~$(V,\emptyset)$, with $\sv=(B,A)$ for $\vs=(A,B)$ and $(A,B)\land (C,D):= (A\cap C, B\cup D)$,%
   \COMMENT{}
   and let ${\vS} \sub {\vU}$ be any submodular separation system in~$\vU$. Let $1\le m\in{\mathbb N}$ and $2\le n\in{\mathbb N}\cup\{\infty\}$.%
   \COMMENT{}
  For these $m$ and~$n$, define
$$\F_m := \big\{\,F\sub\vU : \big|\!\bigcap_{(A,B)\in F}\!\!\! B\>\big| < m\, \big\}$$
\vskip-.5\baselineskip\noindent
and
 $$\F^n_m := \{\,F\in\F_m : |F| < n\,\}.$$%
   \COMMENT{}%
   \COMMENT{}

There is only one small separation in ${\vU}$, the separation~$(\emptyset,V)$. Hence regardless of what~$S$ may be, it has no trivial separation other than~$(\emptyset,V)$.%
   \COMMENT{}
   Since ${\{(V,\emptyset)\}\in\F^n_m}$ for all $m$ and~$n$, this makes $\F^n_m$ standard for~$S$.

Recall that, for any $\F\sub 2^\vS\!$, we write $\F^*$ for the set of stars in~$\F$.

\begin{lemma}\label{Funcrosses}
Let $1\le m\in{\mathbb N}$ and $2\le n\in{\mathbb N}\cup\{\infty\}$. For $\F = \F^n_m$, the $\F^*$-tangles of $S$ are precisely its $\F$-tangles.
\end{lemma}

\begin{proof}
Since $\F^*\sub\F$, it is clear that $\F$-tangles are $\F^*$-tangles. To show the converse, suppose there is an $\F^*$-tangle~$\tau$ that fails to be an $\F$-tangle, because it contains some $F\in\F$ as a subset.

Clearly $F\in\F\setminus\F^*$, so $F$~contains two crossing separations $\vr$ and~$\vs$. Since $\vS$ is submodular, one of their opposite corners $\vr\land\sv$ and $\vs\land\rv$ lies in~$\vS$; let us assume $\vrdash:=\vr\land\sv$ does. Since $\vrdash\le\vr\in\tau$, the consistency of~$\tau$ implies that~$\vrdash$ lies in~$\tau$ (rather than its inverse~$\rvdash$). Indeed, this follows from the definition of consistency if $r'\ne r$. But if $r'=r$ then by $\vrdash\le\vr$ either $\vrdash=\vr\in\tau$ as desired, or $\vrdash < \vr = \rvdash$ with $\vrdash$ small but $\rvdash\in\tau$. Since $(\emptyset,V)$ is the only small separation in~$\vS$ and is in fact trivial, the consistency of~$\tau$ once more implies that $\vrdash\in\tau$.%
   \COMMENT{}

Let $F'$ be obtained from~$F$ by replacing $\vr$ with~$\vrdash$. Note  that $\bigcap_{(A,B)\in F} B = \bigcap_{(A',B')\in F'} B'$ has remained unchanged: although we replaced the set $B$ from $\vr=(A,B)$ in the first intersection with the bigger set~$B'$ from $\vrdash= (A',B')$ in the second, the additional $B'\setminus B$ has empty intersection with the set $D$ from $\vs=(C,D)$, and therefore does not increase the second intersection. Hence our assumption of $F\in\F = \F^n_m$ implies that also $F'\in\F^n_m = \F$.%
   \COMMENT{}

Note that while $\vr$ and~$\vs$ crossed, $\vrdash$~and $\vs$ are nested; indeed, $\{\vrdash,\vs\}$ is a star. Moreover, replacing an element of this star by a smaller separation will yield another star; in particular, it cannot result in another pair of crossing separations. This means that iterating the above uncrossing procedure of replacing in~$F$ an element $\vr\in\tau$ with a smaller separation $\vrdash\in\tau$ in a way that keeps $F$ in~$\F$ will end after at most~$\binom{|F|}{2}$ steps: for every 2-set $\{\vr,\vs\}\sub F$ we will consider only once in this iterated process a pair $\{\vrdash,\vsdash\}$ where $\vrdash$ is either $\vr$ or a replacement of~$\vr$, and $\vsdash$ is either $\vs$ or a replacement of~$\vs$.

Since the above process turns every pair of crossing separations from~$F$ into a 2-star of separations, and a set of separations is a star as soon as all its 2-subsets are stars, the set we turn $F$ into will be a star in~$\F$, an element of~$\F^*$. As it will also still be a subset of~$\tau$, this contradicts our assumption that $\tau$ is an $\F^*$-tangle.
\end{proof}

\begin{lemma}\label{Fshifts}
Let $1\le m\in{\mathbb N}$ and $2\le n\in{\mathbb N}\cup\{\infty\}$. The set $\F^*\!$ of stars in $\F=\F^n_m$ is closed under shifting.
\end{lemma}

\begin{proof}
Suppose that ${\vs}\in\vS$ emulates in~$\vS$ some nontrivial~${\vr}$ not forced by~$\F$. We have to show that for every star $\sigma \sub {\vS} \setminus \{ {\rv} \}$ with $\sigma \in \F^*$ and every ${\vx} \in \sigma$ with ${\vx} \geq {\vr}$ we have $\sigma':=\sigma\!_{{\vx}}^{{\vs}} \in \F^*$.

Let $\vs = (U,W)$, and for $(A,B)\in\sigma$ write $(A',B')\in\sigma'$ for the separation that $(A,B)$ shifts to: if $(A,B) = \vx$ then $(A',B') := (A\cup U,B\cap W)$, while if $(A,B)\in\sigma\setminus\{\vx\}$ then $(A',B') := (A\cap W, B\cup U)$. From these explicit representations of the elements of~$\sigma'$ it is clear that 
 $$\bigcap_{(A',B')\in \sigma'}\!\!\! B'\ = \bigcap_{(A,B)\in \sigma}\!\!\! B'\ \sub \bigcap_{(A,B)\in\sigma}\!\!\! B\,,$$
 since $B'\setminus B\sub U$ for every $(A,B)\in\sigma\setminus\{\vx\}$ while $U\cap B'=\emptyset$ for $(A,B) = \vx$, so that the overall intersection of all the~$B'$ equals that of all the~$B$. And since these sets did not change, nor did their cardinality: as $\sigma\in\F = \F^n_m$ we also have $\sigma'\in\F^n_m = \F$. By Lemma~\ref{shift is star}, this implies~$\sigma'\in\F^*$ as desired.
\end{proof}

Together with Lemmas \ref{Funcrosses} and~\ref{Fshifts}, Theorem~\ref{submodular duality} implies our last main theorem:

     	\begin{FnThm}
     	Let ${\vS} \sub {\vU}$ be a submodular separation system, let $1\le m\in{\mathbb N}$ and $2\le n\in{\mathbb N}\cup\{\infty\}$, and let $\F = \F^n_m$.%
   \COMMENT{}
   Then exactly one of the following two statements holds:
     	\begin{enumerate}[\rm (i)]\itemsep=0pt\vskip-\smallskipamount\vskip0pt
     	\item $S$ has an $\F$-tangle;
     	\item There exists an $S$-tree over $\F^*$.\qed
     	\end{enumerate}
      	\end{FnThm}

     	\end{subsection}

       \begin{subsection}{Phylogenetic trees from tangles of circle separations}\label{sec:phylo}

Finally, let us describe an application of Theorem~\ref{ToTThm} and Theorem~\ref{Fn} in biology. Let $V$ be a set, which we think of as a set of species, or of (possibly unknown) organisms, or of DNA samples. Our aim is to find their Darwinian `tree of life': a~way of dividing $V$ recursively into ever-smaller subsets so that the leaves of this division tree correspond to the individual elements of~$V\!$.

This tree can be formalized by starting with a root node labelled~$\emptyset$ at level~0, a~unique node labelled~$V$ at level~1, and then recursively adding children to each node, labelled~$A\sub V\!$, say, corresponding to the subsets of~$A$ into which $A$ is divided and labelling these children with the subsets to which they correspond. Furthermore, let us label the edges from level~$k-1$ to~$k$ by~$k$.

Every edge~$e$ of this tree naturally defines a bipartition of~$V\!$, to which we assign the label~$k$ of~$e$ as it `order'. Species that are fundamentally different are then separated by a bipartition of low order, while closely related but distinct species are only separated by bipartitions of higher order.

If we draw that tree in the plane, the leaves~-- and hence~$V$~-- will be arranged in a circle, $C$~say. The bipartitions $\{A,B\}$ of~$V$ defined by the edges of the tree are given by {\em circle separations\/}: $A$~and~$B$ are covered by disjoint half-open segments of~$C$ whose union is~$C$.%
   \COMMENT{}

The way in which this tree is found in practice is roughly as follows. One first defines a metric on~$V\!$ in which species $u,v$ are far apart if they differ in many respects. Then one applies some clustering algorithm, such as starting with the singletons $\{v\}$ for $v\in V$ as tiny clusters and then successively amalgamating close clusters (in terms of this distance function) into bigger ones. The bipartitions of $V$ corresponding to pitching a single cluster against the rest of~$V$ will be nested and can therefore be represented as edge-separations of a tree as above, which is then output by the algorithm. If we draw the tree in the plane, the bipartitions then become circle separations of~$V\!$ as earlier.

A~known problem with this approach is that, since every cluster found in this process defines a bipartition of~$V$ that ends up corresponding to an edge of the final tree, inaccuracies in the clustering process immediately affect this tree in an irreversible way. Bryant and Moulton~\cite{BryantMoulton} have suggested a more careful clustering process which produces not necessarily a tree but an outerplanar graph~$G$ on~$V\!$, together with a set of particularly important bipartitions of~$V$ that are not necessarily nested but are still circle separations of~$V\!$ with respect to the outer face boundary of~$G$. The task then is to select from the set $S$ of these bipartition a nested subset that defines the desired phylogenetic tree, or perhaps to generate such a set from~$S$ in some other suitable way such as adding corners of crossing separations already selected.

This is where tangles can help: in a general way, but also in a rather specific way that finds the desired nested set from a set of circular separations of~$V\!$.

Let us just briefly indicate the general way in which tangles can be used for finding phylogenetic trees in a novel way~\cite{TanglesEmpirical}, without the need for any distance-based clustering. As input we need a collection of subsets of~$V\!$ to be used as similarity criteria, such as the set of species~$v\in V$ that can fly or lay eggs, the set of DNA molecules that have base $T$ in position~137, or the set of those organisms that respond to some test in a certain way. Then we define an order function on all the bipartitions of~$V\!$, assigning low order to those bipartitions that do not cut accross many of our criteria sets so as to split them nearly in half. For example, we might count for $s=\{A,B\}$ the number of triples $(a,b,c)$ such that $a\in A$ and $b\in B$ and both $a$ and~$b$ satisfy (are elements of) the criterion~$c$.

This order function is easily seen to be submodular on the universe~$\vU$ of all oriented bipartitions of~$V\!$~\cite{TanglesSocial}. For suitable~$\F$ whose $\F$-tangles are profiles, e.g.\ the $\F=\F^n_m$ with $n>3$ considered in Section~\ref{sec:cluster}, we can then compute the (canonical) tree of tangles as in~\cite{ProfilesNew}, or an $S$-tree over~$\F$ as in~\cite{TangleTreeGraphsMatroids,TangleTreeAbstract} if there are no tangles. In the first case, the tree of tangles for $\F^n_2$ has a good claim to be the phylogenetic tree for the species in~$V\!$, see~\cite{TanglesEmpirical}.

In the concrete scenario of~\cite{BryantMoulton}, we further have the following specific application of tangles as studied in this paper. Let $S$ be the set of circle separations of~$V\!$, taken with respect to its circular ordering found by the current algorithm of~\cite{BryantMoulton}. We can now define $\F$-tangles on this~$S$ just as we did earlier on the set of all bipartitions of~$V\!$, and consider the same order function as earlier.

This order function is not, however, submodular on our restricted set~$S\,$: this would require that $S$ is not just a separation system but a universe of separations, i.e., that corner separations of elements of~$S$ are again in~$S$~-- which is not always the case.%
   \COMMENT{}
   In particular, we do not get a tree-of-tangles theorem for~$S$ from~\cite{ProfilesNew}, or a tangle-tree duality theorem from~\cite{TangleTreeGraphsMatroids,TangleTreeAbstract}.%
   \COMMENT{}%
   \footnote{This is not to say that no submodular order function on $S$ exists that returns the sets~$S_k$ we are interested in as sets~$S_{k'}$ for some other~$k'$. One can indeed construct such a function, but it is neither obvious nor natural.}%
   \COMMENT{}

However, we do get a tree-of-tangles theorem for~$S$ as a corollary of Theorem~\ref{ToTThm}, and a tangle-tree duality theorem as a corollary of Theorem~\ref{Fn}. For this we need the following easy lemma. Let $\vU$ be the universe of all oriented bipartitions of~$V\!$ (see Section~\ref{sec:cluster}), equipped with any submodular order function under which $\{\emptyset,V\}$ has order~0. Let $S$ be the set of all circle separations of~$V\!$ with respect to some fixed cyclic ordering of~$V\!$. 

\begin{lemma}\label{circle submodular}
For every~$k>0$, the set $S_k$ of all circle separations of order~$<k$ is submodular.
\end{lemma}

\begin{proof}
Consider two oriented circle separations $\vr=(A,B)$ and~$\vs=(C,D)$ of~$V\!$. Clearly, $\vr\lor\vs$ is again a circle separation unless the circle segments representing $A$ and~$C$ are disjoint and both segments that join them on the circle meet~$V\!$. In that case, however, the union of the segments representing $B$ and~$D$ is the entire circle, so $\vr\land\vs = (\emptyset,V)$, which is a circle separation.

As $(\emptyset,V)$ has order $0<k$ by assumption, and our order function is submodular on the set of all bipartitions of~$V\!$, this implies that $S_k$ is submodular: given $\vr,\vs\in\vSk$, either both $\vr\lor\vs$ and $\vr\land\vr$ are in~$\vS$ and hence one of them is in~$\vSk$, or one of them is $(\emptyset,V)$ or~$(V,\emptyset)$ and therefore in~$\vSk$.
\end{proof}

Consider any $1\le m\in{\mathbb N}$ and $3 < n\in{\mathbb N}\cup\{\infty\}$.%
   \COMMENT{}
   Let $\F = \F^n_m$ be as defined in Section~\ref{sec:cluster}. Here is our first tree-of-tangles theorem for circle separations:

     	\begin{theorem}\label{CircleToT}
     	For every~$k>0$, the set $S_k$ of circle separations of~$V\!$ of order~$<k$ contains a tree set of separations that distinguishes all the $\F$-tangles of~$S_k$.
      	\end{theorem}

\begin{proof}
    In order to apply Theorem~\ref{ToTThm}, we have to show that all $\F$-tangles of~$S_k$ are abstract tangles, i.e., that they contain no triple $(\vr,\vs,\vt)$ with $\vr\lor\vs\lor\vt = (V,\emptyset)$ (which is the unique co-small separation in~$\vU$). But any such triple lies in~$\F$ for the values of $m$ and~$n$ we specified, so no $\F$-tangle of~$S_k$ contains it.

Lemma~\ref{circle submodular} and Theorem~\ref{ToTThm} thus imply the result.
   \end{proof}

Applying very recent work of Elbracht, Kneip and Teegen~\cite{FiniteSplinters}, we can unify the assertions of Theorem~\ref{CircleToT} over all~$k$, as follows. Let us say that an element $s$ of~$S$ {\em distinguishes\/} two orientations $\rho,\tau$ of subsets of~$S$ if these orient~$s$ differently, i.e., if $s$ has orientations $\vs\in\rho$ and $\sv\in\tau$. If $s$ has minimum order amongst the separations in~$S$ that distinguish $\rho$ from~$\tau$, we say that $s$ distinguishes $\rho$ and~$\tau$ {\em efficiently\/}. Similarly, a set $T\sub S$ {\em distinguishes\/} a set~$\T$ of orientations of subsets of~$S$ {\em efficiently\/} if for every pair of distinct $\rho,\tau\in\T$ there exists an $s\in T$ that distinguishes $\rho$ from~$\tau$ efficiently.

Orientations $\rho$ and~$\tau$ as above are called {\em distinguishable\/} if they are distinguished by some $s\in S$. Note that orientations $\rho$ of~$S_k$ and $\tau$ of~$S_\ell$ for $k\le\ell$ are indistinguishable if and only if $\rho = \tau\cap\vSk$.

Here is our tree-of-tangles theorem for circle separations of mixed order:

\begin{theorem}\label{CircleToTallk}{\rm\cite{FiniteSplinters}}
The set $S$ of all circle separations of~$V\!$ contains a tree set that efficiently distinguishes all the distinguishable $\F$-tangles of subsets $S_k$ of~$S$.
\end{theorem}

Elbracht, Kneip and Teegen~\cite{FiniteSplinters} showed that the tree set in Theorem~\ref{CircleToTallk} can in fact be chosen {\em canonical\/}, i.e., so that every separation system automorphism (see~\cite{AbstractSepSys}) of~$\vS$ acts on~$\vT$ as a set of automorphisms of~$T$.

\goodbreak

Finally, our tangle-tree duality theorem for circle separations:

     	\begin{theorem}\label{CircleTTD}
     	For every~$k>0$, the set $S_k$ of circle separations of~$V\!$ of order~$<k$ satisfies exactly one of the following two assertions:
     	\begin{enumerate}[\rm (i)]\itemsep=0pt\vskip-\smallskipamount\vskip0pt
     	\item $S_k$ has an $\F$-tangle;
     	\item There exists an $S_k$-tree over $\F^*$.
     	\end{enumerate}
      	\end{theorem}

\begin{proof}
   Apply Lemma~\ref{circle submodular} and Theorem~\ref{Fn}.
\end{proof}

       \end{subsection}

     \end{section}

\begin{section}*{Acknowledgement}
     
Geoff Whittle has informed us that his student Jasmine Hall has obtained explicit proofs of Theorems \ref{CircleToT} and~\ref{CircleTTD} by re-working the theory of graph and matroid tangles given in~\cite{BranchDecMatroids,GMX}.

\end{section}

\bibliographystyle{plain}
\bibliography{collective}

     \end{document}